\documentclass{amsart}

\usepackage[utf8]{inputenc} 
\usepackage[ngerman,main=english]{babel}
\usepackage{amsmath,amssymb,amstext,amsthm} 
\usepackage{graphicx}    
\usepackage{enumitem} 

\usepackage[style=alphabetic,citestyle=alphabetic,maxcitenames=4,maxalphanames=4,maxbibnames=10,backref=false,hyperref=true,backend=bibtex,doi=false,isbn=false,url=false,giveninits=true]{biblatex}

\usepackage{csquotes} 

\usepackage[colorlinks=true,linkcolor=black,urlcolor=black,citecolor=black,bookmarksnumbered]{hyperref}

\numberwithin{equation}{section}
\numberwithin{figure}{section}

\newtheorem{theorem}{Theorem}[section]
\newtheorem{lemma}[theorem]{Lemma}
\newtheorem{cor}[theorem]{Corollary}
\newtheorem{prop}[theorem]{Proposition}
\newtheorem{introtheorem}{Theorem}

\newtheorem{introprop}[introtheorem]{Proposition}

\theoremstyle{definition}
\newtheorem{defi}[theorem]{Definition}
\newtheorem{nota}[theorem]{Notation}
\newtheorem*{ass}{Standing Assumption}

\theoremstyle{remark}
\newtheorem{remark}[theorem]{Remark}
\newtheorem{ex}[theorem]{Example}

\newcommand{\SL}[1]{\ensuremath{\mathrm{SL}(#1,\mathbb{R})}}
\newcommand{\Sp}[1]{\ensuremath{\mathrm{Sp}(#1,\mathbb{R})}}
\newcommand{\Id}{\ensuremath{\mathrm{Id}}}
\newcommand{\atheta}{\ensuremath{\mathfrak{a}_\theta}}
\newcommand{\Htwist}{\ensuremath{\mathcal{H}^{\mathrm{Twist}}(\widehat{\lambda};\mathfrak{a}_\theta)}}
\newcommand{\lieg}{\ensuremath{\mathfrak{g}}}
\newcommand{\olambda}{\widehat{\lambda}}
\newcommand{\Buse}[2]{\sigma_{\theta} \left( #1, #2 \right)}    
\newcommand{\norm}[1]{\left\lVert #1 \right\rVert}
\newcommand{\norma}[1]{\left\lVert #1 \right\rVert_{\mathfrak{a}}}
\newcommand{\dG}{\mathrm{d}_G}

\begin{document}

\title{Cataclysms for Anosov representations}
\author{Mareike Pfeil}
\begin{abstract}
In this paper, we construct cataclysm deformations for $\theta$-Anosov representations into a semisimple non-compact connected real Lie group $G$ with finite center, where $\theta \subset \Delta$ is a subset of the simple roots that is invariant under the opposition involution.
These generalize Thurston's cataclysms on Teichm\"uller space and Dreyer's cataclysms for Borel-Anosov representations into $\mathrm{PSL}(n, \mathbb{R})$.
We express the deformation also in terms of the boundary map.
Furthermore, we show that cataclysm deformations are additive and behave well with respect to composing a representation with a group homomorphism. 
Finally, we show that the deformation is injective for Hitchin representations, but not in general for $\theta$-Anosov representations. 
\end{abstract}

\maketitle 

\subsection*{Data availability statement}
Data sharing not applicable to this article as no datasets were generated or analysed during the current study.

\setcounter{tocdepth}{1}
\tableofcontents
\section{Introduction}

Let $S$ be a closed connected orientable surface of genus $g(S)$ at least $2$ and $\pi_1(S)$ its fundamental group.
Representations from $\pi_1(S)$ into a Lie group $G$ can carry information about specific structures on the surface.
A well-known example for this is Teichm\"uller space, the space of marked hyperbolic structures on $S$, whose elements can be identified with discrete and faithful representations from $\pi_1(S)$ into the group $\mathrm{PSL}(2,\mathbb{R})$.
In this paper, we are interested in \emph{Anosov representations} that describe a special dynamical structure on $S$.
They have first been introduced by Labourie in \cite{Labourie}, and his definition was extended by Guichard and Wienhard in \cite{GuichardWienhard_Anosov}.
In recent years, many equivalent characterizations of Anosov representations have been found and the field is an active area of research (see \cite{GGKW}, \cite{KLP_Anosov}, \cite{DGK_ConvexCocompact}, \cite{BPS}, \cite{Zhu_RelativelyDominated19},  \cite{Tsouvalas_ConvexCocompact},  \cite{KasselPotrie_Eigenvalue}, \cite{Zhu_RelativelyDominated21},  \cite{BrayCanaryKaoMartone},\cite{CanaryZhangZimmer_CuspedHitchin}).
Apart from Teichm\"uller space, important examples of Anosov representations are quasi-Fuchsian representations, Hitchin representations, maximal representations (\cite{BILW_MaximalRepresentations}) and  $\theta$-positive Anosov representations (\cite{GuichardWienhard_Positivity}).

Let $G$ be a semisimple Lie group.
An Anosov representation $\rho \colon \pi_1(S)$ is always Anosov \emph{with respect to a parabolic subgroup} $P_\theta^+$ of $G$, where $\theta \subset \Delta$ is a subset of the simple restricted roots. 
We call such a representation \emph{$\theta$-Anosov}.
Hitchin representations, for instance, are $\Delta$-Anosov.
Every $\theta$-Anosov representation $\rho \colon \pi_1(S) \to G$ has an associated boundary map $\zeta \colon \partial_\infty \tilde{S} \to \mathcal{F}_\theta$ into the flag manifold $\mathcal{F}_\theta = G/P_{\theta}^+$.
In this paper, we use a characterization of Anosov representations in terms of the existence of such boundary maps, together with $\theta$-divergence \cite{GGKW}.
In the following, fix a semisimple non-compact connected real Lie group $G$ with finite center and let $\theta \subset \Delta$ be a subset that is invariant under the opposition involution $\iota \colon \Delta \to \Delta$. \\

One approach to understanding Anosov representations is to use techniques that have proven helpful in the case of Teichm\"uller space.
Special deformations of hyperbolic structures, called \emph{earthquakes}, and their generalizations \emph{cataclysms} give insights into the structure of Teichm\"uller space (see \cite{Thurston_Earthquakes}, \cite{Thurston_MinimalStretchMaps} and \cite{Bonahon_Shearing}).
For example, cataclysms can be used to define shearing coordinates, which are closely related  to the symplectic structure on Teichm\"uller space (see \cite{Bonahon_Soezen}).
Thus, cataclysm deformations might also be a tool to understand Anosov representations.
First steps in this direction were taken by Dreyer in \cite{Dreyer_Cataclysms}. 
He generalized Thurston's cataclysm deformations to representations into $\mathrm{PSL}(n, \mathbb{R})$ which are Anosov with respect to the minimal parabolic subgroup.
This a strong assumption, which is not satisfied by most Anosov representations (see \cite{CanaryTsouvalas_Topological_Restrictions_Anosov_representations}). 
In this paper, we generalize Dreyer's construction and define cataclysm deformations for $\theta$-Anosov representations into a semisimple connected non-compact real Lie group $G$ that are Anosov with respect to a parabolic $P_\theta^+$.\\

Dreyer's cataclysms are closely related to Bonahon-Dreyer coordinates on the Hitchin component in $\mathrm{PSL}(n, \mathbb{R})$ (see \cite{Bonahon_Acta}).
Whereas the Hitchin component is well-understood, there is still a lot to learn about Anosov representations outside of the Hitchin component. 
Recent advances in this direction have been made in \cite{LeeLeeStecker}, where the authors study Anosov representations of triangle reflection groups into $\SL{3}$ and describe the set of Anosov representations that do not lie in the Hitchin component.
The cataclysms constructed in this paper provide a tool to also understand Anosov representations that do not lie in the Hitchin component. \\

In order to state the main result we introduce two more concepts.
A \emph{geodesic lamination $\lambda$} on $S$ is a collection of disjoint simple geodesics whose union is closed. 
We denote by $\tilde{\lambda}$ its lift to the universal cover $\tilde{S}$ and $\widehat{\lambda}$ its orientation cover, that is, a two-fold cover with a continuous orientation of the geodesics.
Let $\atheta = \bigcap_{\alpha \in \Delta \setminus \theta} \ker(\alpha)$ and $\Htwist$ be the space of \emph{$\atheta$-valued twisted transverse cycles}.
An element in $\Htwist$ assigns to every oriented arc transverse to $\tilde{\lambda}$ an element in $\atheta$ in a way that is finitely additive, invariant under homotopy respecting the lamination and twisted in the sense that if we change the orientation of an arc, the corresponding element in $\atheta$ changes by the opposition involution $\iota \colon \atheta \to \atheta$.

Our main result is the following:
\begin{introtheorem}[Theorem \ref{thm::cataclysm_deformation_hom}]
\label{thm::cataclysm_intro}
Let $\rho \colon \pi_1(S) \to G$ be a $\theta$-Anosov representation. 
There exists a neighborhood $\mathcal{V}_\rho$ of $0$ in $\Htwist$ and a continuous map 
\begin{align*}
\Lambda \colon \mathcal{V}_\rho &\to \mathrm{Hom}(\pi_1(S), G) \\
\varepsilon &\mapsto \Lambda^\varepsilon \rho
\end{align*}
such that $\Lambda^0 \rho = \rho$.
Up to possibly shrinking the neighborhood $\mathcal{V}_\rho$, the representation $\Lambda^\varepsilon \rho$ is again $\theta$-Anosov.
\end{introtheorem}

The representation $\Lambda^\varepsilon \rho$ is called \emph{$\varepsilon$-cataclysm deformation along $\lambda$ based at $\rho$}.
It depends on the choice of a reference component $P \subset \tilde{S} \setminus \tilde{\lambda}$ and for different choices of $P$, the resulting representations are conjugated. 

For every $\gamma \in \pi_1(S)$, the deformed representation $\Lambda^\varepsilon \rho$ is given by
\begin{align*}
\Lambda^\varepsilon \rho (\gamma) = \varphi^\varepsilon_{P(\gamma P)} \rho(\gamma).
\end{align*}
Here, $P \subset \tilde{S} \setminus \tilde{\lambda}$ is the fixed reference component, $Q$ another component and $\varphi_{PQ}^\varepsilon$ is the \emph{shearing map} between $P$ and $Q$ with respect to the shearing parameter $\varepsilon \in \Htwist$.
For the case of Teichm\"uller space, $\varphi_{PQ}^\varepsilon$ is obtained from a concatenation of hyperbolic isometries that act as stretches along geodesics in the lamination that separate $P$ and $Q$.
In the general case, it is obtained from a concatenation of \emph{stretching maps} $T^H_g$, where $H \in \atheta$ and $g$ is an oriented geodesic between $P$ and $Q$ and $H \in \atheta$. 
We define stretching and shearing maps in Section \ref{sec::stretching_maps} and \ref{sec::cataclysms}, respectively. \\

One way of describing a cataclysm deformation is through the deformation of the associated boundary map.
\begin{introtheorem}[Theorem \ref{thm::boundary_map}]
\label{thm::boundary_intro}
If $\zeta \colon \partial_\infty \tilde{S} \to \mathcal{F}_\theta$ is the boundary map for $\rho$ and $\rho' = \Lambda^\varepsilon \rho$ is the $\varepsilon$-cataclysm deformation of $\rho$ along $\lambda$, then the boundary map $\zeta'$ for $\rho'$ is given by
\begin{align*}
\zeta'(x) = \varphi^\varepsilon_{P Q_x} \cdot \zeta(x),
\end{align*}
where $x$ is a boundary point of the lamination $\tilde{\lambda}$ and $Q_x$ is a component of $\tilde{S} \setminus \tilde{\lambda}$ having $x$ as a vertex.
\end{introtheorem}

\begin{remark}
For the special case of $\Delta$-Anosov representations into $\mathrm{PSL}(n, \mathbb{R})$, Theorems \ref{thm::cataclysm_intro} and \ref{thm::boundary_intro} were proven by Dreyer in \cite{Dreyer_Cataclysms}.
The difference between their and our more general construction for $\theta$-Anosov representations lies in the definition of the parameter space and in the definition of the basic building blocks, the stretching maps.
Further, we consider an arbitrary geodesic lamination $\lambda$, whereas Dreyer's result is for maximal laminations only.
\end{remark}

The first step in the construction of cataclysm deformations is the definition of the parameter space $\Htwist$. 
For the special case that the lamination $\lambda$ is maximal, i.e.\ the complement $\tilde{S} \setminus \tilde{\lambda}$ consists of ideal triangles, we compute that the dimension of $\Htwist$ is 
\begin{align*}
|\theta|(6g(S)-6) + |\theta'|,
\end{align*}
where $\theta' \subset \theta$ is a maximal subset satisfying $\theta' \cap \iota(\theta') = \emptyset$ (see Corollary \ref{cor::dim_Htwist_maximal}).
In the case $G = \mathrm{PSL}(n,\mathbb{R})$ and $\theta = \Delta$, this recovers a result from Dreyer (\cite[Lemma 16]{Dreyer_Cataclysms}).
We determine the dimension of $\Htwist$ more generally for a geodesic lamination $\lambda$ that is not necessarily maximal (Proposition \ref{prop::dim_Htwist}). \\

Having defined cataclysms, we study their properties.
The first observation is that the cataclysm is additive, i.e.\ satisfies a local flow condition.

\begin{introtheorem}[Theorem \ref{thm::additivity}]
The cataclysm deformation is additive in the sense that for two twisted cycles $\varepsilon, \eta \in \Htwist$ sufficiently small,
\begin{align*}
\Lambda^{\varepsilon+\eta} \rho = \Lambda^\varepsilon \left( \Lambda^\eta \rho \right).
\end{align*}
\end{introtheorem}

The second observation is about how a cataclysm behaves with respect to composition of an Anosov representation with a group homomorphism.
Let $\kappa \colon G \to G'$ be a group homomorphism between semisimple Lie groups and denote all objects associated with $G'$ with a prime.
Under assumptions on $\kappa$ and $\theta' \subset \Delta'$, $\kappa \circ \rho$ is $\theta'$-Anosov by a result from Guichard and Wienhard (\cite[Proposition 4.4]{GuichardWienhard_Anosov}). 

\begin{introprop}[Theorem \ref{prop::cataclysm_equivariant}]
\label{thm::equivariant_intro}
If $\kappa \colon G \to G'$ is a group homomorphism between semisimple Lie groups such that $\kappa \circ \rho$ is Anosov, then under the additional assumption that $\kappa_*(\atheta) \subset \mathfrak{a}'_{\theta'}$, we have
\begin{align*}
\Lambda^{\kappa_* \varepsilon} (\kappa \circ \rho) = \kappa \left( \Lambda^\varepsilon \rho\right),
\end{align*}
where $\kappa_* \colon \mathfrak{a} \to \mathfrak{a}'$ is the map between the maximal abelian subalgebras $\mathfrak{a} \subset \mathfrak{g}$ and $\mathfrak{a}' \subset \mathfrak{g}'$ induced by $\kappa$.
\end{introprop}

Finally, we address the question if the cataclysm deformation is injective, i.e.\ if different twisted cycles result in different deformed representations. 

\begin{introprop}[Proposition \ref{prop::shearing_injective}]
\label{prop::shearing_injective_intro}
The map 
\begin{align*}
\mathcal{V}_\rho &\to G^{ \{(P,Q) | P,Q \subset \tilde{S} \setminus \tilde{\lambda} \} } \\
\varepsilon &\mapsto \{\varphi^\varepsilon_{PQ} \}_{(P,Q)}
\end{align*}
that assigns the family of shearing maps to a transverse twisted cycle is injective. 
Here, $\mathcal{V}_\rho \subset \Htwist$ is the neighborhood of $0$ from Proposition \ref{prop::shearing_map}.
\end{introprop}
For the special case $G= \mathrm{PSL}(n, \mathbb{R})$ and $\theta = \Delta$, Proposition \ref{prop::shearing_injective_intro} was observed by Dreyer in \cite[Section 5.1]{Dreyer_Cataclysms}.
One might now be lead to believe that also the cataclysm deformation is injective, which is wrong in general. 
In Subsection \ref{sec::horocyclic}, we consider a family of reducible representations into $\SL{n}$ which we call $(n,k)$-horocyclic and for which we can identify a subspace of $\Htwist$ on which the deformation is trivial.

However, under the assumption that for every connected component $Q \subset \tilde{S} \setminus \tilde{x}$, the intersection of the stabilizers of $\zeta(x)$ for all vertices $x$ of $Q$ is trivial, we show that the cataclysm deformation is injective (Corollary \ref{cor::stab_trivial_injective}).
This result is in particular true for Hitchin representations into $\SL{n}$ (Corollary \ref{cor::injective_SLn}).

\begin{remark}
The results in this paper were obtained in the author's PhD thesis \cite{Pfeil_Cataclysms}.
More detailed versions of the proofs can be found therein.
\end{remark}

\subsection*{Acknowledgments}

I thank my advisors Anna Wienhard and Beatrice Pozzetti for their continuous guidance and support throughout my work on this project.
Further, I thank Francis Bonahon, Guillaume Dreyer, James Farre, Xenia Flamm and Max Riestenberg for helpful discussions and comments, and the anonymous referee for constructive comments.

The author acknowledges funding from the Klaus Tschira Foundation and the Heidelberg Institute for Theoretical Studies, from the Deutsche Forschungs\-gem\-ein\-schaft (DFG, German Research Foundation) – 281869850 (RTG 2229) and under Germany’s Excellence Strategy EXC2181/1-390900948 (the Heidelberg STRUCTURES Excellence Cluster), and from  the U.S. National Science Foundation under Grant No. DMS-1440140 while the author was in residence at the Mathematical Sciences Research Institute in Berkeley, California, during the Fall 2019 semester, and DMS 1107452, 1107263, 1107367 "RNMS: Geometric Structures and Representation Varieties" (the GEAR Network) for a Graduate Internship at the University of Southern California, Los Angeles.

This version of the article has been accepted for publication in \emph{Geometriae Dedicata}, after peer review, but is not the Version of Record and does not reflect post-acceptance improvements, or any
corrections. The Version of Record is available online at: \href{http://dx.doi.org/10.1007/s10711-022-00721-7}{http://dx.doi.org/10.1007/s10711-022-00721-7}.

\section{Preliminaries}

\subsection{Parabolic subgroups of semisimple Lie groups}

In this section, we introduce the notation that we use throughout the paper and give the definition of $\theta$-Anosov representations. 
It is based on \cite[Section 3.2]{GuichardWienhard_Anosov} and \cite[Section 2.2]{GGKW}.

Let $G$ be a connected non-compact semisimple real Lie group and let $\mathfrak{g}$ be its Lie algebra. 
Choose a maximal compact subgroup $K \subset G$. 
Let $\mathfrak{k} \subset \mathfrak{g}$ be the Lie algebra of $K$ and $\mathfrak{p}$ its orthogonal complement with respect to the Killing form.  
Let $\mathfrak{a}$ be a maximal abelian subalgebra in $\mathfrak{p}$ and $\mathfrak{a}^*$ its dual space.
For $\alpha \in \mathfrak{a}^*$, we define 
 \begin{align*}
\mathfrak{g}_\alpha =\{X \in \mathfrak{g} \ | \ \forall H \in \mathfrak{a}: \ \mathrm{ad}_H(X) = \alpha(H)X \},
\end{align*}
where $\mathrm{ad} \colon \mathfrak{g} \to \mathfrak{gl}(\mathfrak{g})$ is the adjoint representation of the Lie algebra $\mathfrak{g}$.
If $\mathfrak{g}_\alpha \neq 0$, we say that $\alpha$ is a \emph{restricted root} with associated \emph{root space} $\mathfrak{g}_\alpha$. 
Let $\Sigma  \subset \mathfrak{a}^* \setminus \{0\}$ be the set of restricted roots. 
Let $\Delta \subset \Sigma$ be a simple system, i.e.\ a subset of $\Sigma$ such that every element in $\Sigma$ can be expressed uniquely as a linear combination of elements in $\Delta$ where all coefficients are either non-negative or non-positive (see \cite[§II.6]{Knapp}). 
The elements of $\Delta$ are called \emph{simple roots}.
Denote by $\Sigma^+$ the set of \emph{positive roots}, namely the set of all elements in $\Sigma$ which have non-negative coefficients with respect to the generating set $\Delta$, and by $\Sigma^-$ the \emph{negative roots}. 

Define the subalgebras
\begin{align*}
\mathfrak{n}^+ = \bigoplus_{\alpha \in \Sigma^+} \mathfrak{g}_\alpha, \qquad \mathfrak{n}^- = \bigoplus_{\alpha \in \Sigma^+} \mathfrak{g}_{-\alpha }
\end{align*}
and consider the corresponding subgroups  $N^+ = \exp(\mathfrak{n}^+)$ and $N^- = \exp(\mathfrak{n}^-)$. 
Further, let $A := \exp(\mathfrak{a})$ and let $Z_K(\mathfrak{a})$ be the centralizer of $\mathfrak{a}$ in $K$.
The \emph{standard minimal parabolic subgroup} is the subgroup $B^+ := Z_K(\mathfrak{a})AN^+$, and in the same way, one defines $B^- := Z_K(\mathfrak{a})AN^-$. 
A subgroup of $G$ that is conjugate to $B^+$ is called \emph{minimal parabolic subgroup} or \emph{Borel subgroup}.
The group $N^+$ is the \emph{unipotent radical} of $B^+$, that is, the largest normal subgroup consisting of unipotent elements.
A subgroup $P \subset G$ is called \emph{parabolic} if it is  conjugate to a subgroup containing $B^+$. 
Parabolic subgroups can be classified by subsets of the set of simple roots $\Delta$ as follows:
Given a subset $\theta \subset \Delta$, let
\begin{align}
\label{eq::atheta}
\atheta := \bigcap_{\alpha \in \Delta \setminus \theta} \mathrm{ker}(\alpha) \ \subset \mathfrak{a},
\end{align}
and let $Z_K(\atheta)$ be its centralizer in $K$. 
A \emph{standard parabolic subgroup} is a subgroup of the form $P_\theta^\pm := Z_K(\atheta)AN^\pm$.  
These are the only parabolic subgroups containing $B^+$, and every parabolic subgroup $P$ is conjugate to some $P_\theta^+$ for a unique $\theta \subset \Delta$ \cite[Proposition 5.14]{BorelTits}.
Thus, conjugacy classes of parabolic subgroups are in one-to-one correspondence with subsets $\theta \subset \Delta$.
For example, for $\theta = \Delta$ we have $\mathfrak{a}_\Delta = \mathfrak{a}$, so $\Delta$ corresponds to the conjugacy class of the minimal parabolic subgroup $B^+$. 

Let $\Sigma_{\theta}^+$ denote the set of all elements in $\Sigma^+$ that do not belong to the span of $\Delta \setminus \theta$. 
Define the Lie subalgebras
\begin{align*}
\mathfrak{n}_\theta^+ = \bigoplus_{\alpha \in \Sigma_\theta^+} \mathfrak{g}_\alpha, \qquad \mathfrak{n}_\theta^- = \bigoplus_{\alpha \in \Sigma_\theta^+} \mathfrak{g}_{-\alpha} 
\end{align*}
and  $N_\theta^\pm := \exp(\mathfrak{n}_\theta^\pm)$. 
The group $N_\theta^\pm$ is the \emph{unipotent radical} of $P_\theta^\pm$, and 
the standard parabolic subgroups $P_\theta^\pm$ are equal to the semidirect product $L_\theta N_\theta^\pm$, where
$L_\theta := P_\theta^+ \cap P_\theta^-$ is the common \emph{Levi subgroup} of $P_\theta^+$ and $P_\theta^-$.

The quotient $\mathcal{F}_\theta = G / P_{\theta}^+$ is called \emph{flag manifold}. 
It is a compact $G$-homogeneous space and its elements are called \emph{flags}. 
The group $G$ acts on $\mathcal{F}_\theta$ by left-multiplication, which will be denoted by $(g,F) \mapsto g \cdot F$.
There is a one-to-one correspondence between $\mathcal{F}_\theta$ and the set of parabolic subgroups conjugate to $P_\theta^+$: When $F =  g \cdot P_\theta^+ \in \mathcal{F}_\theta$ is a flag, then its stabilizer is a parabolic subgroup conjugate to $P_\theta^+$. 
We will often make no distinction between elements in $\mathcal{F}_\theta$ and parabolic subgroups conjugate to $P_\theta^+$.

The \emph{Weyl group} of $G$ is 
\begin{align*}
W := N_K(\mathfrak{a}) / Z_K(\mathfrak{a}),
\end{align*}
where $N_K(\mathfrak{a})$ is the normalizer of $\mathfrak{a}$ in $K$.
Seen as a subgroup of $\mathrm{GL}(\mathfrak{a})$, it is a finite Coxeter group with system of generators given by the orthogonal reflections in the hyperplanes $\mathrm{ker}(\alpha) \subset \mathfrak{a}$ for $\alpha \in \Delta$.
The \emph{Weyl chambers} are the connected components of $\mathfrak{a} \setminus \bigcup_{\alpha \in \Delta} \mathrm{ker}(\alpha)$. 
A closed Weyl chamber is the closure of a Weyl chamber.
The set of positive roots $\Sigma^+$ singles out the \emph{closed positive Weyl chamber} of $\mathfrak{a}$ defined by
\begin{align*}
\overline{\mathfrak{a}^+} := \{ H \in \mathfrak{a} \ | \ \alpha(H) \geq 0 \ \forall \alpha \in \Sigma^+\}.
\end{align*}
The Weyl group $W$ acts simply transitively on the Weyl chambers and thus there exists a unique element $w_0 \in W$ such that $w_0( \overline{\mathfrak{a}^+}) = - \overline{\mathfrak{a}^+}$.

\begin{defi}
\label{def::opposition_involution}
The involution  on $\mathfrak{a}$ defined by $\iota := - w_0$ is called \emph{opposition involution}. 
It induces a dual map on $\mathfrak{a}^*$ which is also denoted by $\iota$ and defined by $\iota(\alpha) = \alpha \circ \iota$ for all $\alpha \in \Delta$.
\end{defi}

Consider the space $\mathcal{F}_\theta \times \mathcal{F}_{\iota(\theta)}$ of pairs of flags. 
There exists a unique open orbit for the diagonal action of $G$ on $\mathcal{F}_\theta \times \mathcal{F}_{\iota(\theta)}$ (see \cite[Section 2.1]{GuichardWienhard_Anosov}), and one representative of this orbit is the pair $(P_\theta^+, P_\theta^-)$.
Note that $P_{\theta}^-$ is an element of $\mathcal{F}_{\iota(\theta)}$ since $P_\theta^-$ is conjugate to $P_{\iota(\theta)}^+$ by $w_0$.
\begin{defi}
Two parabolic subgroups $P^+$ and $ P^-$ are \emph{transverse} if the pair $(P^+, P^-)$ lies in the unique open $G$-orbit of $\mathcal{F}_\theta \times \mathcal{F}_{\iota(\theta)}$ for some $\theta \subset \Delta$.
Given $P^+ \in \mathcal{F}_\theta$ and $P^- \in \mathcal{F}_{\iota(\theta)}$, we say that $P^+$ is \emph{transverse} to $P^-$ (and $P^-$ is transverse to $P^-$) if the pair $(P^+,P^-)$ is transverse.
\end{defi}
Note that the subset $\theta$ in this definition is already determined by the parabolic $P^+$, since every parabolic subgroup is conjugate to $P_\theta^+$ for a unique $\theta \subset \Delta$, and that the unipotent radical $N_\theta^+$ acts simply transitively on parabolics transverse to $P_\theta^+$.
In the following, we restrict our attention to subsets $\theta \subset \Delta$ for which $P_\theta^+$ and $P_\theta^-$ are conjugate. 
This is true if and only if $\iota(\theta) = \theta$.
\begin{ass}
For the rest of this paper, we assume that $\theta \subset \Delta$ is invariant under the opposition involution $\iota$, i.e.\ $\iota(\theta) = \theta$.
\end{ass}

\begin{ex}
\label{ex::SLn}
Consider the special linear group $G = \SL{n}$ with maximal compact subgroup $K =\mathrm{SO}(n, \mathbb{R})$.
The Lie algebra $\mathfrak{sl}(n, \mathbb{R})$ is given by all traceless $n \times n$ matrices, and as maximal abelian subalgebra $\mathfrak{a}$ in $\mathfrak{p}$ we choose the set of diagonal traceless matrices.
Let $\lambda_i \in \mathfrak{a}^*$ be the evaluation at the $i$-th entry of an element in $\mathfrak{a}$, i.e.\ $\lambda_i (\mathrm{diag}(a_1, \dots, a_n)) = a_i$.
A short calculation shows that the root space $\mathfrak{g}_{\alpha}$ for 
$\alpha \in \mathfrak{a}^*$ is non-zero if and only if $\alpha = \alpha_{ij} := \lambda_i - \lambda_j$ for some $1 \leq i \neq j \leq n$. 
The root space $\mathfrak{g}_{\alpha_{ij}}$ is spanned by the matrix $E_{ij}$ that has a $1$ in the $i$-th row and $j$-th column and $0$ everywhere else.
The set of restricted roots $\Sigma$, a system of simple roots $\Delta$ and the set of positive roots $\Sigma^+$ are
\begin{align*}
\Sigma  & =\{ \alpha_{ij} \ | \ 1 \leq i \neq j \leq n \}, \\
\Delta &= \{ \alpha_{i(i+1)}| \ i = 1, \dots, n-1 \} \ \mathrm{and}\\
\Sigma^+ &= \{ \alpha_{ij}| \ 1 \leq i < j \leq n \}.
\end{align*}
To shorten notation, set $\alpha_{i} := \alpha_{i(i+1)}$. 
We often identify the set of simple roots $\Delta$ with the set $\{ 1, \dots, n-1\}$.
The subalgebra $\mathfrak{n}^+$ and sugroup $N^+$ are given by the upper triangular matrices with only $0$s and only $1$s on the diagonal, respectively.
The minimal parabolic subgroup $B^+$ is the set of all upper triangular matrices, i.e.\
\begin{align*}
B^+ = Z_K(\mathfrak{a}) AN^+ = \left\{
\begin{pmatrix}
* & \dots  & * \\  & \ddots &\vdots  \\  0 & & * 
\end{pmatrix} \middle| \  \mathrm{det} = 1  \right\}.
\end{align*}
Similarly, $B^-$ is the set of all lower triangular matrices.

Let $\theta = \{i_1, \dots, i_k \} \subset \Delta$ with  $1 \leq i_1 < \dots < i_k \leq n-1$.  
Also, let $i_0 := 0$ and $i_{k+1} := n$. 
For $j=1, \dots, k+1$, let $m_j := i_j - i_{j-1}$, so the $m_j$ describe the sizes of the gaps between the elements in $\theta$.
For $i \in \Delta$, we have that $\mathrm{ker}(\alpha_i)$ is given by all traceless diagonal matrices where the $i$th and $(i+1)$th  entry are equal. 
Thus $\atheta = \bigcap_{\alpha \in \Delta \setminus \theta} \mathrm{ker}(\alpha)$ is given by block diagonal matrices, where the $j$-th block is a scalar multiple of the $m_j \times m_j$ identity matrix, and the centralizer $Z_K(\atheta)$ is given by block diagonal matrices where the blocks $A_{m_j}$ are $m_j \times m_j$ matrices.

The standard parabolic subgroup $P_{\theta}$ for $\theta$ is given by upper block triangular matrices. More precisely, 
\begin{align*}
P_\theta^+ = Z_K(\atheta)AN^+ = \left\{
\begin{pmatrix}
A_{m_1} &* & \dots  & * \\ 
0 & \ddots & & \vdots \\ 
\vdots & & \ddots & * \\  
0 & \dots & 0 & A_{m_{k+1}} 
\end{pmatrix} \in \SL{n} \right \},
\end{align*}
where the $A_{m_j}$ have size $m_j \times m_j$.
Similarly, $P_\theta^-$ consists of lower block triangular matrices with blocks of sizes $m_j \times m_j$.
The unipotent radical $N_\theta^\pm$ is the subgroup of $P_\theta^\pm$ with all blocks on the diagonal being identity matrices and the Levi subgroup $L_\theta$ is the group of block diagonal matrices with the blocks of sizes $m_j \times m_j$.

Elements of the flag manifold $\mathcal{F}_\theta$ are families of nested subspaces of the form
\begin{align*}
F =  \left(F^{(i_1)} \subset \dots \subset F^{(i_k)} \subset \mathbb{R}^n \right) \quad \mathrm{with} \ \mathrm{dim}\  F^{(i_j)} = i_j,
\end{align*}
so flags in the sense of linear algebra. 
This explains the name \emph{flag manifold}.

The maximal abelian subalgebra $\mathfrak{a}$ is the set of traceless diagonal matrices, so it is isomorphic to $\mathbb{R}^{n-1}$. 
The Weyl group is isomorphic to the symmetric group $\mathcal{S}_n$ with generators the orthogonal reflections along the hyperplanes $\mathrm{ker}(\alpha_i)$ for $i = 1, \dots, n-1$.
Here, the hyperplane  $\mathrm{ker}(\alpha_i)$ is given by all elements $\mathrm{diag}(a_1, \dots, a_n) \in \mathfrak{a}$ with $a_i = a_{i+1}$.
The closed positive Weyl chamber is
\begin{align*}
\overline{\mathfrak{a}^+} 
= \left \{ \mathrm{diag}(a_1, \dots, a_n) \  \middle| \  \sum_{i=1}^n a_i= 0, a_i \geq a_{i+1} \forall i\right \}.
\end{align*}
The longest element $w_0$ of the Weyl group, which sends $\overline{\mathfrak{a}^+}$ to $-\overline{\mathfrak{a}^+}$ acts on $\mathfrak{a}$ by reversing the order of the diagonal elements.
On $\Delta$, the longest element acts as $w_0(\alpha_{i}) = - \alpha_{n-i}$, so for the opposition involution $\iota \colon \Delta \to \Delta$, we have $\iota(\alpha_i) = \alpha_{n-i}$. 
The assumption that $\iota(\theta) = \theta$ in this setting means that $i \in \theta$ if and only if $n-i \in \theta$.
Here, we use the identification of $\Delta$ with $\{1, \dots, n-1\}$.
Two flags $F =  \left(F^{(i_1)} \subset \dots \subset F^{(i_k)}\right)$ and $E =  \left(E^{(i_1)} \subset \dots \subset E^{(i_k)} \right)$ in $\mathcal{F}_{\theta}$ are transverse if and only if $F^{(i_j)} \oplus E^{(n-i_j)} =\mathbb{R}^n$ for all $i_j \in \theta$.
\end{ex}

\subsection{Anosov representations}

Anosov representations are representations from $\pi_1(S)$ to $G$ with special dynamical properties. 
They were first defined by Labourie in \cite{Labourie} and the definition was extended by Guichard and Wienhard in \cite{GuichardWienhard_Anosov}.
Today, many equivalent characterization of Anosov representations exist, we present a characterization given in \cite[Theorem 1.3]{GGKW}. 
We are only interested in surface groups, but Anosov representations can be defined more generally for any word-hyperbolic group $\Gamma$. 
\begin{ass}
Throughout this paper, let $S$ be a closed connected oriented surface of genus at least $2$. 
Fix an auxiliary hyperbolic metric $m$ on $S$.
Denote by $\tilde{S}$ the universal cover of $S$, which then carries a hyperbolic metric as well, and let $\partial_\infty \tilde{S}$ be the boundary at infinity.
Fixing base points on $S$ and $\tilde{S}$, the fundamental group $\pi_1(S)$ of $S$ acts on $\tilde{S}$ by isometries. 
\end{ass}

Before giving the definition of Anosov representations, we need to introduce some more concepts. 
Let $\theta \subset \Delta$ be a subset of the simple roots and $\mathcal{F}_\theta = G/P_{\theta}^+$ the flag manifold for the standard parabolic subgroup $P_\theta^+$.
\begin{defi}
Let $\rho \colon \pi_1(S) \to G$ be a representation and $\zeta \colon \partial_\infty \tilde{S} \to \mathcal{F}_\theta$ be a map.
\begin{itemize}
\item $\zeta$ is called \emph{$\rho$-equivariant} if $\zeta(\gamma x) = \rho(\gamma) \cdot  \zeta(x)$ for all $x \in \partial_\infty \tilde{S}$ and $\gamma \in \pi_1(S)$.
\item  $\zeta$ is called \emph{transverse} if for every pair of points $x \neq y \in \partial_\infty \tilde{S}$, the images $\zeta(x), \zeta(y) \in \mathcal{F}_\theta$ are transverse.
\item $\zeta$ is called \emph{dynamics-preserving} for $\rho$ if for every non-trivial element $\gamma \in \pi_1(S)$, its unique attracting fixed point in $\partial_\infty \tilde{S}$ is mapped to an attracting fixed point of $\rho(\gamma)$ on $\mathcal{F}_\theta$.
\end{itemize}
Further, $\rho$ is called \emph{$\theta$-divergent} if for all $\alpha \in \theta$ we have $\alpha(\mu(\rho(\gamma))) \to \infty$  as the word length of $\gamma \in \pi_1(S)$ goes to infinity.
\end{defi}
\begin{defi}[{\cite[Theorem 1.3]{GGKW}}]
\label{def::Anosov}
A representation $\rho \colon \pi_1(S) \to G$ is \emph{$\theta$-Anosov} if it is $\theta$-divergent and there exists a continuous $\rho$-equivariant map $\zeta \colon \partial_\infty \tilde{S} \to \mathcal{F}_\theta$ that is  transverse and dynamics-preserving.
The map $\zeta$ is called the \emph{boundary map} for $\rho$.
\end{defi}
The set of $\theta$-Anosov representations forms an open subset of the representation variety $\mathrm{Hom}(\pi_1(S), G)$ \cite[Theorem 5.13]{GuichardWienhard_Anosov}.

If there exists a continuous, dynamics-preserving boundary map $\zeta \colon \partial_\infty \tilde{S} \to \mathcal{F}_\theta$ for $\rho$, then it is unique and $\rho$-equivariant. 
Further, the boundary map is H\"older continuous \cite[Theorem 6.1]{Bridgeman_ea_Pressure_metric}.

\subsection{Geodesic laminations}

Geodesic laminations are an important tool in low-dimensional topology. 
Although the following Definition \ref{def::geodesic_lamination} makes use of the auxiliary hyperbolic metric $m$ on $S$, geodesic laminations can be defined independent of the choice of a metric (see \cite[Section 3]{Bonahon_Transverse_Hoelder}).
An overview of geodesic laminations can be found in \cite{Bonahon_Laminations}.
\begin{defi}
\label{def::geodesic_lamination}
A \emph{geodesic lamination} $\lambda$ on $S$ is a collection of simple complete disjoint geodesics such that their union is a closed subset of $S$.
The geodesics contained in the lamination $\lambda$ are called \emph{leaves}. 
A geodesic lamination $\lambda$ is \emph{maximal} if every connected component of $S \setminus \lambda$ is isometric to an ideal triangle.
We denote by $\tilde{\lambda}$ the lift of $\lambda$ to the universal cover $\tilde{S}$.
A lamination $\lambda$ is \emph{connected} if it cannot be written as a disjoint union of two sublaminations, i.e. subsets that are themselves laminations.
It is \emph{maximal} if the complement $S \setminus \lambda$ consists of ideal triangles. 
\end{defi}

We can equip the leaves of a lamination $\lambda$ with an orientation -- but if the lamination is maximal, this cannot be done in a continuous way. 
For this reason, we look at the \emph{orientation cover} $\olambda$ of $\lambda$.
\begin{defi}
\label{def::orientation_cover}
The \emph{orientation cover} $\olambda$ is a $2$-cover of the lamination $\lambda$ whose geodesics are oriented in a continuous fashion. 
\end{defi}
For example, if $\lambda$ is orientable, then $\olambda$ consists of two disjoint copies of $\lambda$ with opposite orientations.
In order to consider arcs transverse to the lamination $\olambda$, we need an ambient surface $\widehat{U}$ in which $\olambda$ is embedded. 
Let $U \subset S$ be an open neighborhood of $\lambda$ that avoids at least one point in the interior of each ideal triangle in $S \setminus \lambda$. 
The orientation cover $\olambda \to \lambda$ extends to a cover $\widehat{U} \to U$. 
We denote by $\mathfrak{R} \colon \widehat{U} \to \widehat{U}$ the orientation reversing involution.
There is a one-to-one correspondence between oriented arcs transverse to $\lambda$ and non-oriented arcs transverse to $\lambda$: Every oriented arc transverse to $\lambda$ has a unique lift transverse to $\olambda$ such that its intersection with $\olambda$ is positive. 

We will be interested in arcs transverse to a lamination $\lambda$ that are well-behaved in the following sense.
\begin{defi}
\label{def::tightly_transverse}
A arc $k$ is \emph{tightly transverse} to $\lambda$ if it is simple, compact, transverse to $\lambda$, has non-empty intersection with $\lambda$, and every connected component $d \subset k \setminus \lambda$ either contains an endpoint of $k$ or the positive and the negative endpoint of $d$ lie on different leaves of the lamination. 
We use the same terminology for arcs $\tilde{k}$ transverse to the universal cover $\tilde{\lambda}$ or to the orientation cover $\olambda$ of $\lambda$.
\end{defi}
\begin{nota}
\label{nota::lamination}
Throughout the paper, we use the following notation: 
By $\lambda$, we denote a lamination on $S$, by $\tilde{\lambda}$ its lift to the universal cover $\tilde{S}$ and by $\olambda$ the orientation cover. 
We denote by $k$, $\tilde{k}$ and $\widehat{k}$ an oriented arc tightly transverse to $\lambda$, $\tilde{\lambda}$ and $\olambda$, respectively. 
Since the universal cover $\tilde{S}$ is oriented, the leaves of $\tilde{\lambda}$ intersecting $\tilde{k}$ have a well-defined transverse orientation determined by $\tilde{k}$, i.e. if $g$ is a leaf of $\tilde{\lambda}$ intersecting $\tilde{k}$, we orient $g$ such that  the intersection $\tilde{k} \pitchfork g$ is positive. 
Denote by $P$ and $Q$ the connected components of $\tilde{k}$ containing the negative and positive endpoint of $\tilde{k}$, respectively.
Let $\mathcal{C}_{PQ}$ be the set of connected components of $\tilde{S} \setminus \tilde{\lambda}$ separating $P$ and $Q$.
Likewise, for two geodesics $g$ and $h$ in $\tilde{\lambda}$, let $\mathcal{C}_{gh}$ be the set of all connected components in $\tilde{S} \setminus \tilde{\lambda}$ lying between $g$ and $h$.
If $R \subset \tilde{S} \setminus \tilde{\lambda}$ is in $\mathcal{C}_{PQ}$, denote by $g_{R}^0$ and $g_{R}^1$ the oriented geodesics passing through the negative and positive endpoint of $\tilde{k} \cap R$, respectively (Figure \ref{fig::Notation_g_R}).
The fact that $\tilde{k}$ is tightly transverse guarantees that the geodesics $g_R^0$ and $g_R^1$ are disjoint.
\end{nota}

\begin{figure}
\centering
\includegraphics[width=0.8\textwidth]{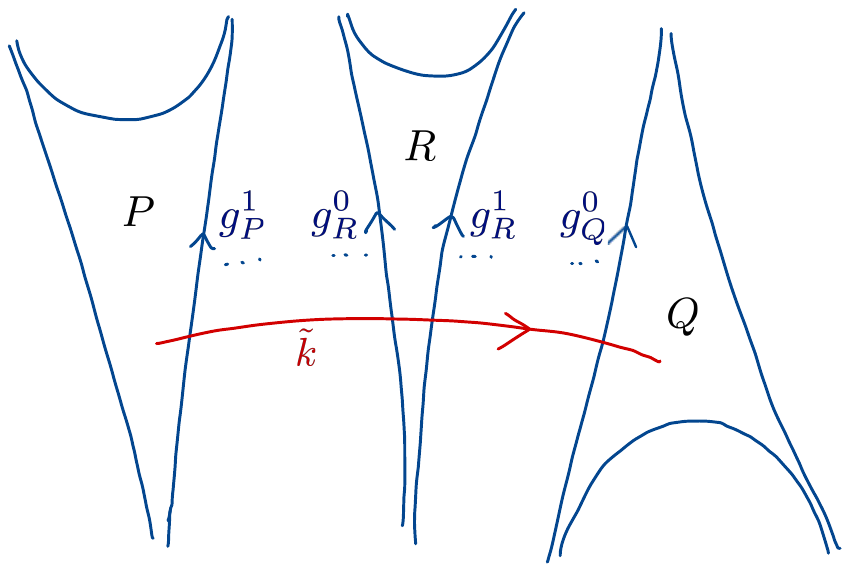}
\caption{For an arc $\tilde{k}$ tightly transverse to $\tilde{\lambda}$ and an ideal triangle $R \in \mathcal{C}_{PQ}$, denote by $g_R^0$, and $g_R^1$ the oriented geodesics in $\tilde{\lambda}$ passing through the negative and  positive endpoint of $\tilde{k} \cap R$, respectively. }
\label{fig::Notation_g_R}
\end{figure}

The following is classical property of geodesic laminations. 
\begin{lemma}[{\cite[Lemma 5.3]{Bonahon_Acta}}]
\label{lem::divergence_radius}
Let $g$ and $h$ be two geodesics in $\tilde{\lambda}$.
There is a function $r \colon \mathcal{C}_{gh} \to \mathbb{N}$, called \emph{divergence radius}, and constants $C_1, C_2, A_1, A_2>0$ such that the following conditions hold:
\begin{enumerate}
\item $C_1 e^{-A_1r(R)} \leq \ell(\tilde{k} \cap R) \leq C_2 e^{-A_2r(R)}$ for every $R \in \mathcal{C}_{gh}$;
\item for every $N \in \mathbb{N}$, the number of triangles $R \in \mathcal{C}_{gh}$ with $r(R) = N$ is uniformly bounded, independent of $N$.
\end{enumerate}
Here, $\ell$ denotes the length function on $\tilde{S}$ induced by the fixed metric on $S$.
\end{lemma}

Roughly speaking, the divergence radius gives the minimal number of lifts of the projection of $R$ to $S$ that separate $R$ from either $g$ or $h$. 
It is defined as follows:
Let $\tilde{k}$ be an arc transverse to $\tilde{\lambda}$ from $g$ to $h$ as above. 
The definition does not depend on the choice of $\tilde{k}$.
Let $p \colon \tilde{S} \to S$ be the projection and $k := p(\tilde{k})$ the projection of $\tilde{k}$ to $S$. 
Consider the connected component $d_R := p(\tilde{k} \cap R) \subset p(\tilde{k}) \setminus \lambda$.
The leaves $g_{R}^0$ and $g_{R}^1$ of $\tilde{\lambda}$ project to leaves $p(g_R^0)$ and $p(g_R^1)$ of passing through the endpoints of $d_R$ (see Figure \ref{fig::divergence_radius}).
The arc $d_R$ divides $p(R)$ into two regions.
Consider all components $d \subset k \setminus \lambda$ that have their negative endpoint on $p(g_{R}^0)$ and their positive endpoints on $p(g_{R}^1)$. 
Since $k$ is compact, at least one of the two regions of $p(R)$ defined by $d_R$ contains only finitely many such components $d \subset k \setminus \lambda$.
Define $r(R) \in \mathbb{N}$ as the minimal number of components of $k \setminus \lambda$ contained in one of the regions that have their endpoints on $p(g_{R}^0)$ and $p(g_{R}^1)$.
\begin{figure}
\centering
\includegraphics[width=0.8\textwidth]{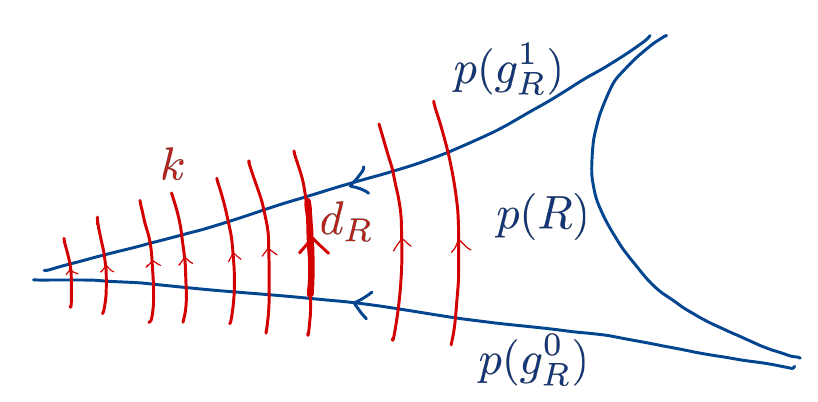}
\caption{This picture illustrates the definition of divergence radius and shows the situation on the surface $S$. The compact arc $k = p(\tilde{k})$ crosses the connected component $p(R)$ of $S \setminus \lambda$ several times. 
The arc $d_R := p(\tilde{k} \cap R) \subset k \setminus \lambda$ divides the component $p(R)$ into two regions.  
The divergence radius $r(R)$ is the minimal number of components of $d \subset k \setminus \lambda$ that have their endpoints on the same geodesics $g_{R}^0$ and $g_{R}^1$ as $d_R$. 
In this picture, we have $r(R) = 2$.}
\label{fig::divergence_radius}
\end{figure}

\begin{ex}
If $R \in \mathcal{C}_{gh}$ is the only lift of $p(R)$ crossed by $\tilde{k}$, then $r(R) = 0$.
If $\tilde{k}$ meets $R$ and $R'$ that project to the same component $p(R) = p(R')$ in $S$, and if $\tilde{k}$ does not meet any other lift of $p(R)$ between $R$ and $R'$, then $|r(R) - r(R')| \leq 1$.
If $\tilde{k}$ crosses only finitely many components of $\tilde{S} \setminus \tilde{\lambda}$, then the divergence radius $r(R)$ is bounded over all $R \in \mathcal{C}_{gh}$.
\end{ex}

A detailed treatment of the divergence radius can be found in \cite[Section 1]{Bonahon_Shearing}.

\section{Stretching maps}
\label{sec::stretching_maps}

Stretching maps are the basic building blocks for cataclysm deformations.
In the case of Teichm\"uller space, they are hyperbolic isometries that act as a stretch along a geodesic. 
In the general case of $\theta$-Anosov representations, they can be thought of as a generalization of such a stretch. 
Stretching maps depend on an oriented geodesic which determines the direction of the stretch, and on an element in $\atheta$ which determines the amount of stretching.

Let $\rho \colon \pi_1(S) \to G$ be a fixed $\theta$-Anosov representation with boundary map $\zeta \colon \partial_\infty \tilde{S} \to \mathcal{F}_\theta$. 
For an oriented geodesic $g$ in $\tilde{S}$ with positive and negative endpoint $g^+$ and $g^-$, respectively, let $P_g^\pm := \zeta(g^\pm)$ be the flags associated to the endpoints $g^\pm$.

\begin{defi}
\label{def::stretch_map}
Let $g$ be an oriented geodesic in $\tilde{S}$ and $H \in \mathfrak{a}_\theta$. 
Let $(P_g^+, P_g^-)$ be the pair of transverse parabolics associated with $g$ and let $m_g \in G$ such that $m_g \cdot P_\theta^\pm= P_g^{\pm}$. 
We define the \emph{$H$-stretching map along $g$} as
\begin{align*}
T^H_g := m_g\exp(H)m_g^{-1}.
\end{align*}
\end{defi}
Note that the element $m_g$ in the definition of the stretching map is only unique up to an element in $L_\theta = P^+_\theta \cap P^-_\theta$, but since $\exp(\atheta)$ lies in the centralizer of $L_\theta$, the definition of the stretching map is well-defined.
\begin{ex}
For the case of $\SL{n}$, the pair $(P_g^+, P_g^-)$ associated with an oriented geodesic $g$ determines a subspace splitting of $\mathbb{R}^n$.
With respect to a basis adapted to this splitting, the stretching map $T_g^H$ is a diagonal matrix, so it acts as stretch on the subspaces. 
This explains the name \emph{stretching map}.
\end{ex}

\begin{lemma}
\label{lem::stretch_properties}
Let $g$ be an oriented geodesic and denote by $\overline{g}$ the geodesic with opposite orientation.
For $H, H_1, H_2 \in \atheta$, the stretching map has the following properties:
\begin{enumerate}
\item $T^{H_1}_g T^{H_2}_g = T^{H_1 + H_2}_g$,
\item $\left(T^H_g\right)^{-1} = T^{-H}_g$,
\item \label{item::orientation} $T^H_{\overline{g}} = T^{-\iota(H)}_g$ and 
\item\label{item::equivariance} $\rho$-equivariance,  i.e. $T^H_{\gamma g} = \rho(\gamma)T^H_g \rho(\gamma)^{-1}$ for all $\gamma \in \pi_1(S)$.
\end{enumerate}
\end{lemma}
\begin{proof}
These properties follow from short computations, using that the longest element of the Weyl group maps $P_\theta^\pm$ to $P_\theta^\mp$ and the $\rho$-equivariance of the boundary map $\zeta$ for \eqref{item::orientation} and  \eqref{item::equivariance}, respectively.
\end{proof}

We say that two geodesics $g$ and $h$ are \emph{separated by wedges} if for every connected component $R \in \mathcal{C}_{gh}$, the geodesics $g_R^0$ and $g_R^1$ share an endpoint.
We estimate the distance between two stretching maps for different oriented geodesics $g$ and $h$ in terms of the distance between $g$ and $h$. 
To do so, fix a metric $\dG$ on $G$ that is left-invariant and almost right-invariant. 
More precisely,, $\dG$ satisfies for all $a_1, a_2, b \in G$ that 
\begin{align*}
\dG(a_1b, a_2b) \leq \norm{\mathrm{Ad}_{b^{-1}}}_{\mathrm{op}(\lieg)} \dG(a_1, a_2),
\end{align*}
where  $\norm{\cdot}_{\mathrm{op}(\lieg)}$ is the operator norm induced by a fixed norm $\norm{\cdot}_{\lieg}$ coming from a scalar product on $\lieg$.
For the existence of such a metric, see \cite[Lemma A]{SchiffShnider}.
\begin{prop}
\label{prop::Close_stretches}
There exist constants $C,A>0$, depending on $\tilde{k}$ and $\rho$, such that for every $H \in \atheta$ and for all geodesics $g,h$ in $\tilde{\lambda}$ that intersect $\tilde{k}$, are oriented positively with respect to the orientation of $\tilde{k}$ and are separated by wedges, we have
\begin{align*}
\dG(T^H_g, T^H_h) \leq C \left((e^{\norma{H}}+1) \mathrm{d}(g,h)^A \right), 
\end{align*}
where $\norma{\cdot}$ is a suitable norm on $\mathfrak{a}$.
\end{prop}

For the proof of the proposition, we make use of \emph{slithering maps}, which are explicit elements in $G$ that relate the pairs of parabolics associated to $g$ and $h$.
Slithering maps were first introduced for Teichm\"uller space in \cite[Section 2]{Bonahon_Shearing} and for the case of Hitchin representations into $\mathrm{PSL}(n,\mathbb{R})$ in \cite[Section 5.1]{Bonahon_Acta}.
The authors already mention that their construction possibly extends to a more general context. 
Indeed, we have the following result:

\begin{prop}
\label{prop::slithering}
Let the lamination $\lambda$ be maximal. 
There exists a unique family $\{\Sigma_{gh}\}_{(g,h)}$ of elements in $G$, indexed by all pairs of leaves $g, h$ in $ \tilde{\lambda}$, that satisfies the following conditions:
\begin{enumerate}[label=(\arabic*), ref=\arabic*]
\item \label{cond::slithering_composition} $\Sigma_{gg} = \mathrm{Id}$, $\Sigma_{hg} = \left(\Sigma_{gh}\right)^{-1}$, and $\Sigma_{gh'} = \Sigma_{gh} \Sigma_{hh'}$ when one of the three geodesics $g,h,h'$ separates the others;
\item \label{cond::slithering_Hoelder} $\Sigma_{gh}$ depends locally separately H\"older continuously on $g$ and $h$, i.e.\ there exist constants $C,A >0$ depending on $\tilde{k}$ and $\rho$ such that
$\dG\left( \Sigma_{gh}, Id\right) \leq C\mathrm{d}(g,h)^A$.
\item \label{cond::slithering_triangle} if $g$ and $h$ are oriented and have a common positive endpoint $g^+ = h^+ \in \partial_{\infty} \tilde{S}$ and if $g^-$ and $h^-$ are the negative endpoints of $g$ and $h$, respectively, then $\Sigma_{gh}$ is the unique element in the unipotent radical of $P_g^+$ that sends $P_{h}^-$ to $P_{g}^-$.
\end{enumerate}
From \eqref{cond::slithering_composition}-\eqref{cond::slithering_triangle} it follows that if $g$ and $h$ are oriented in parallel, $\Sigma_{gh}$ sends the pair $(P_{h}^+, P_{h}^-)$ to the pair $(P_{g}^+, P_{g}^-)$.
\end{prop}
Here, we say that two geodesics $g,h \subset \tilde{\lambda}$ are \emph{oriented in parallel} if exactly one of the orientations of $g$ and $h$ agrees with the boundary orientation of the connected component of $\tilde{S} \setminus (g \cup h)$ that separates $g$ from $h$.

The proof of Proposition \ref{prop::slithering} is analogous to the case of Hitchin representations in \cite[Section 5.1]{Bonahon_Acta}, so we omit it here and only sketch how $\Sigma_{gh}$ is constructed.
For the case that the two geodesics $g$ and $h$ are adjacent, the slithering map $\Sigma_{gh}$ is defined as in \eqref{cond::slithering_triangle}.
For $R \in \mathcal{C}_{gh}$, the geodesics $g_R^-$ and $g_R^-$ as in Figure \ref{fig::Notation_g_R} share an endpoint, so we can define $\Sigma_{R} := \Sigma_{g_R^- g_R^+}$. 
From these basic slithering maps, we can construct $\Sigma_{gh}$ as follows: 
Let $\mathcal{C} = {R_1, \dots, R_m} \subset \mathcal{C}_{gh}$ be a finite subset of components separating $g$ and $h$, labeled from $g$ to $h$.
Set $\Sigma_{\mathcal{C}} := \Sigma_{R_1} \circ \dots \circ \Sigma_{R_m}$. 
Then $\Sigma_{gh}$ is the limit of the elements $\Sigma_{\mathcal{C}}$ as $\mathcal{C}$ tends to $\mathcal{C}_{gh}$.
Note that the H\"older continuity of the boundary map $\zeta$ is crucial to ensure that the limit exists. 

\begin{remark}
\label{rem::slithering_wedges}
The basic slithering map $\Sigma_R$ exists whenever the geodesics $g_R^-$ and $g_R^+$ share an endpoint.
Thus, we can construct the slithering map $\Sigma_{gh}$ also for an arbitrary geodesic lamination $\lambda$, under the assumption that the two geodesics $g$ and $h$ are separated by wedges.
\end{remark}

Using the slithering map, we can prove Proposition \ref{prop::Close_stretches}.
\begin{proof}[Proof of Proposition \ref{prop::Close_stretches}]
Without loss of generality we can assume that  the pair of transverse parabolics $(P_g^+, P_g^-)$ associated with $g$ agrees with the pair $(P_\theta^+, P_\theta^-)$ of standard transverse parabolics. 
Then the slithering map $\Sigma_{h g}$ sends $P_\theta^\pm$ to $P_{h}^\pm$.
Note that the slithering map exists by Remark \ref{rem::slithering_wedges}, since by assumption, $g$ and $h$ are separated by wedges.
Thus, in the definition of the stretching maps, we can choose $m_{h} := \Sigma_{h g}$.
Using left-invariance and almost right-invariance of $\dG$, we have
\begin{align*}
\nonumber \dG(T^H_g, T^H_h) &= \dG(\exp(H), \Sigma_{hg} \exp(H) \Sigma_{gh})\\
&\leq  \norm{\mathrm{Ad}_{\exp(-H)}}_{\mathrm{op}(\mathfrak{g})} \dG(\Id, \Sigma_{hg}) + \dG(\Id, \Sigma_{gh}) .
\end{align*}
By H\"older continuity of the slithering map, there exist constants $C, A>0$ depending on $\tilde{k}$ and $\rho$ such that 
\begin{align*}
\dG(\Id, \Sigma_{gh}) = \dG(\Id, \Sigma_{hg}) \leq C \mathrm{d}(g, h)^A.
\end{align*}
In remains to estimate  $\norm{\mathrm{Ad}_{\exp(-H)}}_{\mathrm{op}(\lieg)}$.
We can express an element in $\mathrm{op}(\lieg)$ as matrix in $\mathrm{GL}(\lieg) \subset \mathfrak{gl}(\lieg)$, where we choose a basis of $\lieg$ adapted to the root space decomposition $\mathfrak{g} = \lieg_0 \bigoplus_{\alpha \in \Sigma} \lieg_{\alpha}$.
Consider the infinity norm $\norm{\cdot}_\infty$ on $\mathfrak{gl}(\lieg)$, that is given by the maximal absolute row sum of the matrix. 
We have $\mathrm{Ad}_{\exp(-H)} = \exp(\mathrm{ad}_{(-H)})$, which is a diagonal matrix with entries $0$ and $e^{\alpha(H)}$ for $\alpha \in \Sigma$, possibly with multiplicity. 
Define a norm on $\mathfrak{a}$ by $\norm{H}_{\mathfrak{a}} := \max_{\alpha \in \Sigma} |\alpha(H)|$.
By equivalence of norms, we have
\begin{align*}
\norm{\mathrm{Ad}_{\exp(-H)}}_{\mathrm{op}(\lieg)}   \leq B \norm{\mathrm{Ad}_{\exp(-H)}}_\infty \leq B \max_{\alpha \in \Sigma} e^{\norma{H}}
\end{align*} 
for some constant $B>0$.
Combining this with the above estimates proves the claim.
\end{proof}
The following corollary covers the special case of Proposition \ref{prop::Close_stretches} when the two geodesics bound the same connected component in $\tilde{S} \setminus \tilde{\lambda}$. 

\begin{cor}
\label{cor::adjacent_stretches}
Let $\tilde{k}$ be an oriented arc transverse to $\tilde{\lambda}$, let $R \subset \tilde{S} \setminus \tilde{\lambda}$ such that $\tilde{k} \cap R \neq \emptyset$ and let $r(R)$ be the divergence radius (see Lemma \ref{lem::divergence_radius}).
There exist constants $C, A >0$ depending on $\tilde{k}$ and $\rho$ such that
\begin{align*}
\dG\left(T^H_{g_R^0} T^{-H}_{g_R^1}, \Id\right) \leq C \left( e^{ \norma{H}}+1 \right) e^{-Ar(R)}. 
\end{align*}
\end{cor}
\begin{proof}
By compactness of $\tilde{k}$, there exists a constant $B>0$ depending on $\tilde{k}$ and $\rho$ such that $\mathrm{d}(g_R^0, g_R^1)\leq B \ell(\tilde{k} \cap R)$. 
The Corollary now follows from Proposition \ref{prop::Close_stretches}, using the left-invariance of the metric and  Lemma \ref{lem::divergence_radius}.
\end{proof}

\section{Transverse twisted cycles}

The amount of deformation of a cataclysm is determined by a so-called transverse cycle for the lamination $\lambda$. 
For the special case $G = \SL{2}$, these cycles take values in $\mathbb{R}$.
For the general case of $\theta$-Anosov representations, they take values in $\atheta$. 
Recall from \eqref{eq::atheta} that $\atheta = \bigcap_{\alpha \in \Delta \setminus \theta} \ker(\alpha)$. 

\begin{defi}
\label{def::transverse_cycle}
An \emph{$\mathfrak{a}_\theta$-valued transverse cycle for $\lambda$} is a map associating to each unoriented arc $k$ transverse to $\lambda$ an element $\varepsilon(k) \in \mathfrak{a}_\theta$, which satisfies the following properties:
\begin{enumerate}
\item $\varepsilon$ is finitely additive, i.e. $\varepsilon(k) = \varepsilon(k_1) + \varepsilon(k_2)$ if we split $k$ in two subarcs $k_1$, $k_2$ with disjoint interiors.
\item $\varepsilon$ is $\lambda$-invariant, i.e. $\varepsilon(k) = \varepsilon(k')$ whenever the arcs $k$ and $k'$ are homotopic via a homotopy respecting the lamination $\lambda$.
\end{enumerate}
We denote the vector space of $\mathfrak{a}_\theta$-valued transverse cycles for $\lambda$ by $\mathcal{H}(\lambda; \atheta)$.
\end{defi}

Instead of transverse cycles for the lamination $\lambda$, we focus on transverse cycles for the orientation cover $\olambda$. 
Since there is a bijective correspondence between transverse arcs for $\olambda$ and oriented transverse arcs for $\lambda$, this allows us to assign values to oriented transverse arcs transverse for $\lambda$.
Further, we require the cycles to satisfy a twist condition.

\begin{defi}
\label{def::twisted_cycles}
The space of \emph{$\mathfrak{a}_{\theta}$-valued transverse twisted cycles} for the orientation cover $\olambda$ is 
\begin{align}
\label{eq::HTwist}
\mathcal{H}^{\mathrm{Twist}}(\olambda; \atheta) := \{  \ \varepsilon \in \mathcal{H}(\olambda; \atheta) \ | \ \mathfrak{R}^
*\varepsilon =  \iota(\varepsilon) \ \},
\end{align}
where $\iota \colon \atheta \to \atheta$ is the opposition involution and $\mathfrak{R} \colon \widehat{U} \to \widehat{U}$ is the orientation-reversing involution for the orientation cover $\olambda$.
\end{defi}
In the following, if not stated otherwise, a transverse twisted cycle will always be an $\atheta$-valued transverse twisted cycle.

We can apply a transverse twisted cycle $\varepsilon \in \Htwist$ to a pair $(P,Q)$ of connected components $P,Q$ of $\tilde{S} \setminus \tilde{\lambda}$ as follows:
Recall that we can identify oriented arcs transverse to $\lambda$ with unoriented arcs transverse to the orientation cover $\olambda$. 
Let $\tilde{k}_{PQ}$ be a transverse oriented arc in $\tilde{S}$ from $P$ to $Q$, let $k_{PQ}$ be its projection to $S$ and $\widehat{k}_{PQ}$ the unique lift of $k_{PQ}$ such that the intersection of $\widehat{k}_{PQ}$ and $\olambda$ is positive. 
We define
\begin{align}
\label{eq::varepsilon_PQ}
\varepsilon(P,Q) := \varepsilon(\widehat{k}_{PQ}).
\end{align}
By the twist condition, we have $\varepsilon(Q,P) = \iota\left( \varepsilon(P,Q) \right)$.

\begin{remark}
The twist condition is motivated by fact that we use the transverse cycles as parameters for the stretching maps introduced in Section \ref{sec::stretching_maps}.
It guarantees that they behave nicely under reversing the orientation of an arc. 
More precisely, for two adjacent components $P$ and $Q$ separated by a geodesic $g$, oriented to the left as seen from $P$, we have with Lemma \ref{lem::stretch_properties},
\begin{align*}
T_{\overline{g}}^{\varepsilon(Q,P)} = T_g^{-\iota(\varepsilon(Q,P))} = T_g^{-\varepsilon(P,Q)} = \left(T_g^{\varepsilon(P,Q)}\right)^{-1}. 
\end{align*}
A similar behavior is inherited by the shearing maps that we will define in Section \ref{sec::shearing_maps}.
\end{remark}

We now compute the dimension of the space of transverse twisted cycles.

\begin{prop} 
\label{prop::dim_Htwist}
Let $G$ be a connected non-compact semisimple  real Lie group, $\theta \subset \Delta$ and let $\theta' \subset \theta$ be a maximal subset satisfying $\iota(\theta') \cap \theta' = \emptyset$. 
Then 
\begin{align*}
\mathrm{dim} \ \mathcal{H}^{\mathrm{Twist}}(\olambda; \atheta) =  |\theta| \left(-\chi(\lambda) + n_o(\lambda)\right) + | \theta'|\left(n(\lambda) - n_o(\lambda)\right),
\end{align*}
where $\chi(\lambda)$ is the Euler characteristic of $\lambda$, $n(\lambda)$ is the number of connected components of $\lambda$ and $n_o(\lambda)$ is the number of components that are orientable.
\end{prop}
\begin{proof}
Let $\varepsilon \in \Htwist$ and let $\{ H_\alpha\}_{\alpha \in \theta}$ be a basis for $\atheta$ satisfying $\iota(H_\alpha) = H_{\iota(\alpha)}$ for all $\alpha \in \theta$.
Write $\varepsilon = \sum_{\alpha \in \theta} \varepsilon_\alpha H_\alpha$, where $\varepsilon_\alpha \in \mathcal{H}(\olambda; \mathbb{R})$ is an $\mathbb{R}$-valued transverse cycle. 
The twist condition is equivalent to $\mathfrak{R}^*\varepsilon_\alpha = \varepsilon_{\iota(\alpha)}$ for all $\alpha$.
We can split up $\varepsilon_\alpha$ as $\varepsilon_\alpha = \varepsilon_\alpha^+ + \varepsilon_\alpha^-$, where $\varepsilon_\alpha^{\pm}$ lies in the $(\pm 1)$-eigenspace $\mathcal{H}(\olambda; \mathbb{R})^\pm$ of $\mathfrak{R}^*$.
By the twist condition, we have $\varepsilon_{\iota(\alpha)}^+ = \varepsilon_\alpha^+$ and $\varepsilon_{\iota(\alpha)}^- = - \varepsilon_\alpha^{-}$. 
Thus, we have
\begin{align*}
\mathrm{dim} \ \mathcal{H}^{\mathrm{Twist}}(\olambda; \atheta) &= |\theta'| \ \mathrm{dim} \ \mathcal{H}(\widehat{\lambda}; \mathbb{R}) + | \mathrm{Fix}(\iota) \cap \theta |  \ \mathcal{H}(\widehat{\lambda}; \mathbb{R})^+,
\end{align*}
where $\mathrm{Fix}(\iota) \cap \theta$ is the set of elements in $\theta$ that are fixed under $\iota$.
By a result from Bonahon \cite[Proposition 1]{Bonahon_Shearing}, the dimension of $\mathcal{H}(\lambda; \mathbb{R})$ is  $-\chi(\lambda)+ n_o(\lambda)$ and the dimension of $\mathcal{H}(\olambda; \mathbb{R})$ is $-2\chi(\lambda)+n_o(\lambda)+n(\lambda)$.
Using the one-to-one correspondence between arcs transverse to $\olambda$ and oriented arcs transverse to $\lambda$, we can identify $\mathcal{H}(\widehat{\lambda}; \mathbb{R})^+$  with $\mathcal{H}(\lambda; \mathbb{R})$, so $\mathrm{dim} \mathcal{H}(\olambda; \mathbb{R})^+$ is equal to $-\chi(\lambda) + n_o(\lambda)$.
The claim now follows from a short computation, taking into account that $| \theta |  =  | \mathrm{Fix}(\iota) \cap \theta  | + 2|\theta'|$.
\end{proof}
\begin{cor}
\label{cor::dim_Htwist_maximal}
If the lamination $\lambda$ is maximal, we have
\begin{align*}
\mathrm{dim} \ \mathcal{H}^{\mathrm{Twist}}(\olambda; \atheta) =  |\theta| (6\cdot g(S)-6) + | \theta'|.
\end{align*}
\end{cor}

We conclude this section with an estimate for transverse cycles.
Let $\tilde{k}$ be an oriented arc transverse to $\tilde{\lambda}$ between two connected components $P$ and $Q$ of $\tilde{S} \setminus \tilde{\lambda}$.

Let $\norm{\cdot}_{\atheta}$ be the maximum norm with respect to a basis $\{H_\alpha\}_{\alpha \in \theta}$ for $\atheta$, i.e.\ if $X = \sum_{\alpha \in \theta} x_\alpha H_\alpha$ with $x_\alpha \in \mathbb{R}$, then $\norm{X}_{\atheta} = \max_{\alpha \in \theta} |x_\alpha|$. 
Further, on $\mathcal{H}(\olambda; \atheta)$ consider the norm $\norm{\varepsilon}_{\mathcal{H}(\olambda; \atheta)} = \max_{\alpha \in \theta} \norm{\varepsilon_\alpha}_{\mathcal{H}(\olambda; \mathbb{R})}$ for a transverse cycle $\varepsilon = \sum_{\alpha \in \theta} \varepsilon_\alpha H_\alpha$ for some fixed norm on $\mathcal{H}(\olambda; \mathbb{R})$.

\begin{lemma}
\label{lem::cycle_estimate}
There exists some constant $C > 0$, depending on $\tilde{k}$, such that for every transverse cycle $\varepsilon \in \mathcal{H}(\olambda;\atheta)$, for every $R \in  \mathcal{C}_{PQ}$, 
\begin{align*}
\norm{\varepsilon(P,R)}_{\atheta} \leq C \norm{\varepsilon}_{\mathcal{H}(\olambda; \atheta)} (r(R) +1).
\end{align*}
\end{lemma}

\begin{proof}
In \cite[Lemma 6]{Bonahon_Shearing}, the statement is proven for $\mathbb{R}$-valued transverse cycles. 
The more general case of $\atheta$-valued transverse cycles follows from this by writing $\varepsilon = \sum_{\alpha \in \theta} \varepsilon_\alpha H_\alpha$ and the definition of the norms.
\end{proof}

\section{Cataclysm deformations}
\label{sec::cataclysms}

\subsection{Shearing maps}
\label{sec::shearing_maps}

To a $\theta$-Anosov representation $\rho \colon \pi_1(S) \to G$ and a transverse twisted cycle $\varepsilon \in \Htwist$, we now assign a family $\{\varphi_{PQ}^\varepsilon\}_{(P,Q)}$ of elements in $G$, called \emph{shearing maps}, where $(P,Q)$ ranges over all pairs $(P,Q)$ of connected components of $\tilde{S} \setminus \tilde{\lambda}$.

Let $\mathcal{C} = \{ R_1, \dots, R_m\} \subset \mathcal{C}_{PQ}$ be a finite subset of connected components of $\tilde{S} \setminus \tilde{\lambda}$ that lie between $P$ and $Q$, labeled from $P$ to $Q$.
Recall that we can apply $\varepsilon$ to a pair of connected components of $\tilde{S} \setminus \tilde{\lambda}$ as in \eqref{eq::varepsilon_PQ}.
Define
\begin{align*}
\varphi_{\mathcal{C}}^\varepsilon := \left( T_{g_1^-}^{\varepsilon(P, R_1)}  T_{g_1^+}^{-\varepsilon(P, R_1)}  \right) \dots \left( T_{g_m^-}^{\varepsilon(P, R_m)}  T_{g_m^+}^{-\varepsilon(P, R_m)}  \right)  T_{g_Q^-}^{\varepsilon(P, Q)},
\end{align*}
where $g_i^\pm := g_{R_i}^{\pm}$ (see Figure \ref{fig::Notation_g_R}).
\begin{prop}
\label{prop::shearing_map}
Let $\tilde{k}$ be an arc transverse to $\tilde{\lambda}$ connecting $P$ to $Q$.
There exists a constant $B >0$ depending on $\tilde{k}$  and  the representation $\rho$ such that for $\varepsilon \in \Htwist$ with $\norm{\varepsilon} _{\mathcal{H}^{\mathrm{Twist}}(\widehat{\lambda};\mathfrak{a}_\theta)} < B$, the limit
\begin{align*}
\varphi^{\varepsilon}_{PQ} := \lim_{\mathcal{C} \to \mathcal{C}_{PQ}} \varphi^{\varepsilon}_{\mathcal{C}}
\end{align*}
exists.
\end{prop}
\begin{defi}
For $P,Q \in \tilde{S} \setminus \tilde{\lambda}$, the element $\varphi_{PQ}^\varepsilon \in G$ is called the \emph{shearing map} from $P$ to $Q$ with respect to the shearing parameter $\varepsilon\in \Htwist$.
\end{defi}
Note that the shearing maps also depend on the representation $\rho$, which is not reflected in the notation.
If two $\theta$-Anosov representations $\rho, \rho'$ are conjugated, then also the corresponding shearing maps are conjugated by the same element.

\begin{proof}[Proof of Proposition \ref{prop::shearing_map}]
Let $\mathcal{C} \subset \mathcal{C}_{PQ}$ be as above. 
Define
\begin{align*}
\psi_{\mathcal{C}}^\varepsilon := \varphi_{\mathcal{C}}^{\varepsilon} T_{g_Q^-}^{-\varepsilon(P,Q)}.
\end{align*}
The first step of the proof is to show that $\dG(\psi_{\mathcal{C}}^\varepsilon, \Id)$ is uniformly bounded, the bound depending on $\tilde{k}$ and $\rho$.
Without loss of generality assume that for all $R \in \mathcal{C}_{PQ}$, $g_R^0$ and $g_R^1$ share an endpoint. 
We can do so because there are only finitely many components that do not have this property, thus this just changes the uniform bound by an additive constant.
By the triangle inequality and left-invariance of $\dG$, we have
\begin{align*}
\dG(\psi^{\varepsilon}_{\mathcal{C}}, \Id)
&\leq \sum_{j=1}^{m} \dG\left(T^{\varepsilon(P,R_j)}_{g_j^0}  T^{-\varepsilon(P,R_j)}_{g_j^1},  \Id \right) \\
&\leq  C_1 \sum_{j=1}^{m} \left(e^{\norma{\varepsilon(P,R_j)}} +1 \right) e^{-Ar(R_j)}, 
\end{align*}
where for the last inequality, we use the estimate from  Corollary \ref{cor::adjacent_stretches}.
By Lemma \ref{lem::cycle_estimate} and the fact that the number of connected components with fixed divergence radius is uniformly bounded by some $D \in \mathbb{N}$ (Lemma \ref{lem::divergence_radius}), it follows that there exists $C_2>0$ such that
\begin{align*}
 \sum_{j=1}^{m_1} \left(e^{\norma{\varepsilon(P,R_j)}} +1 \right) e^{-Ar(R_j)} 
  &\leq C_2 \sum_{j=1}^{m_1} \left( e^{C\norm{\varepsilon}(r(R_j)+1)} +1 \right) e^{-Ar(R_j)}  \\
 &\leq C_2D \sum_{r=0}^\infty \left( e^{C\norm{\varepsilon}(r+1)}  +1 \right) e^{-Ar}.
\end{align*}
For $\norm{\varepsilon} < A/C$, this sum converges.
Thus, $\dG(\psi^{\varepsilon}_{\mathcal{C}}, Id)$ is uniformly bounded, the bound depending on $\tilde{k}$ and $\rho$.

To show that the limit $\lim_{\mathcal{C} \to \mathcal{C}_{PQ}} \varphi_{\mathcal{C}}^\varepsilon$ exists, 
choose  a sequence $(\mathcal{C}_m)_{m \in \mathbb{N}}$ of subsets of  $\mathcal{C}_{PQ}$ such that $\mathcal{C}_m$ has  cardinality $m$ and such that $\mathcal{C}_m \subset \mathcal{C}_{m+1}$ for all $m \in \mathbb{N}$.
Fix $m \in \mathbb{N}$ and let $R \subset \tilde{S} \setminus \tilde{\lambda}$ be such that $\mathcal{C}_{m+1} = \mathcal{C}_{m} \cup \{ R\}$.
Further, let  $\mathcal{C}, \mathcal{C'} \subset $ be such that $\mathcal{C}_m = \mathcal{C} \  \dot{\cup} \  \mathcal{C'}$, $\mathcal{C}_{m+1} = \mathcal{C} \  \dot{\cup} \ \{ R\} \ \dot{\cup} \  \mathcal{C'}$ and such that $R$ separates the components in $\mathcal{C}$ from the components in $\mathcal{C'}$.
By the triangle inequality, left-invariance and almost-right invariance of the metric, we have
\begin{align*}
\dG\left(\psi_{\mathcal{C}_m}^{\varepsilon}, \psi_{\mathcal{C}_{m+1}}^{\varepsilon}\right) &\leq \norm{\mathrm{Ad}_{\psi_{\mathcal{C'}^\varepsilon}}}_{\mathrm{op}(\lieg)} \ \dG\left(\Id, T^{\varepsilon(P, R)}_{g_R^0} T^{-\varepsilon(P, R)}_{g_R^1} \right) \\
&\leq C  \left( e^{ \norma{\varepsilon(P,R)}}+1 \right) e^{-Ar(R)}.
\end{align*}
For the last estimate, we use that $\dG\left(\psi_{\mathcal{C}'}^{\varepsilon}, \Id \right)$ is uniformly bounded and Corollary \ref{cor::adjacent_stretches}.
As seen above in the proof of uniform convergence, this goes to $0$ as $r(R)$ goes to infinity. 
It follows that $\left( \psi^{\varepsilon}_{\mathcal{C}_m} \right)_{m \in \mathbb{N}}$ is a Cauchy sequence, so converges.
Thus, also $\left( \varphi^{\varepsilon}_{\mathcal{C}_m} \right)_{m \in \mathbb{N}}$ converges as $m$ goes to infinity.
\end{proof}
Note that, at this stage, the bound on $\norm{\varepsilon}$ depends on the transverse arc $\tilde{k}$ between $P$ and $Q$. 
We will see in Proposition \ref{prop::varphi_PQ_final} that the bound can be made independent on the transverse arc $\tilde{k}$.

The family of shearing maps depends continuously on the twisted cycle $\varepsilon$.
Further, it has the following properties.
\begin{prop}
\label{prop::shearing_composition}
For connected components $P,Q,R \in \tilde{S} \setminus \tilde{\lambda}$, $\varepsilon \in \Htwist$ small enough and $\gamma \in \pi_1(S)$, the shearing maps satisfy
\begin{itemize}
\item  $ \left(\varphi^\varepsilon_{PQ}\right)^{-1} =  \varphi^\varepsilon_{QP}$,
\item $\varphi^\varepsilon_{PQ} = \varphi^\varepsilon_{PR} \ \varphi^\varepsilon_{RQ}$ and
\item $\varphi^\varepsilon_{(\gamma P)( \gamma Q)} = \rho(\gamma) \varphi^\varepsilon_{PQ}  \rho(\gamma)^{-1}$.
\end{itemize}
\end{prop}

\begin{figure}
\centering
\includegraphics[width=0.7\textwidth]{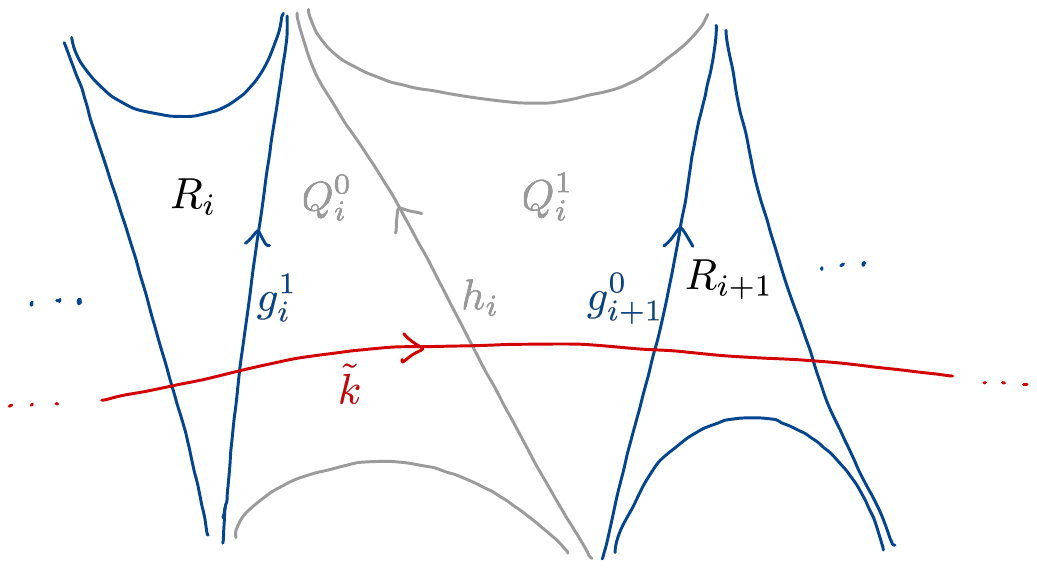}
\caption{In between two triangles $R_i$ and $R_{i+1}$, insert an auxiliary geodesic $h_i$ and two triangles $Q_i^0$ and $Q_i^1$ to approximate the part of the lamination $\tilde{\lambda}$ between $R_i$ and $R_{i+1}$ by a finite lamination.}
\label{fig::Notation_Q_i}
\end{figure}
\begin{proof}
The $\rho$-equivariance follows from the $\rho$-equivariance of the stretching maps (Lemma \ref{lem::stretch_properties}).
For the other properties, assume first that $P$ and $Q$ are separated by finitely many leaves of $\tilde{\lambda}$. 
In this case, $\mathcal{C}_{PQ} = \{R_1, \dots, R_m \}$ and every $R_i$ is adjacent to $R_{i+1}$. 
Then the behavior under taking the inverse follows from the behavior of $T_g^H$ under taking inverses and the twist condition of the cycle.
The composition property follows from the definition of $\varphi^\varepsilon_{PQ}$ and the additivity of the cycle. 
For the general case and a finite subset $\mathcal{C} \subset \mathcal{C}_{PQ}$, we approximate the part of the lamination $\tilde{\lambda}$ between the components in $\mathcal{C}$ by a finite lamination as in \cite[Lemma 5.8]{Bonahon_Acta}:
If $\mathcal{C} = \{R_1, \dots, R_m\}$, for every $i = 1, \dots, m$, we introduce an auxiliary geodesic $h_i$ and two triangles $Q_i^0$ and $Q_i^1$ that separate $R_i$ and $R_{i+1}$ as shown in Figure \ref{fig::Notation_Q_i}.
This gives a finite sequence of components $Q_0^0, Q_0^1, R_1, Q_1^0, Q_1^1, R_2, \dots, Q_m^0, Q_m^1, R_m$ separating $P$ from $Q$, where each component is adjacent to the next one.
Note that the geodesics $h_i$ are not contained in the lamination $\tilde{\lambda}$, and the $Q_i^{0/1}$ are not elements of $\mathcal{C}_{PQ}$.
We define a finite concatenation of maps of the form $\left( T_{g_R^-}^{\varepsilon(P, R)}  T_{g_R^+}^{-\varepsilon(P, R)}  \right)$, with $R$ being equal to either $R_i$ or $Q_i^{0/1}$ for some $i$.
We denote this composition by $\widehat{\varphi}_{\mathcal{C}}^\varepsilon$. 
For $\widehat{\varphi}_{\mathcal{C}}^\varepsilon$, we can show the proposition as explained above. 
Let $\widehat{\psi}_\mathcal{C}^\varepsilon := \widehat{\varphi}_{\mathcal{C}}^\varepsilon T^{-\varepsilon(P,Q)}_{g_Q^0}$.
Using the same techniques as in the proof of Proposition \ref{prop::shearing_map}, namely triangle inequality, almost-invariance of the distance, Proposition \ref{prop::Close_stretches} and the fact that $\psi_{\mathcal{C}}^\varepsilon$ is uniformly bounded, we obtain the estimate
\begin{align*}
\dG\left( \widehat{\psi}_\mathcal{C}^\varepsilon, \psi_{\mathcal{C}}^\varepsilon \right) \leq C \sum_{R \in \mathcal{C}_{PQ} \setminus \mathcal{C}} e^{-Ar(R)}
\end{align*} 
for some constants $C$ and $A$.
The right-hand side converges to $0$ as $\mathcal{C}$ converges to $\mathcal{C}_{PQ}$.
In total, this shows that $\widehat{\psi}_\mathcal{C}^\varepsilon$ and $\psi_{\mathcal{C}}^\varepsilon$ have the same limit for $\mathcal{C}$ to $\mathcal{C}_{PQ}$.
This implies that also $\widehat{\varphi}_{\mathcal{C}}^\varepsilon$ and $\varphi_{\mathcal{C}}^\varepsilon$ have the same limit $\varphi_{PQ}^\varepsilon$. 
Thus, $\varphi_{PQ}^\varepsilon$ inherits the behavior under inversion and taking inverses of $\widehat{\varphi}_{\mathcal{C}}^\varepsilon$.
For a more detailed version of the proof, see \cite[Proposition 5.6]{Pfeil_Cataclysms}.
\end{proof}

Up to now, the bound on the cycle $\varepsilon \in \Htwist$ that guarantees the convergence depends on the transverse arc $\tilde{k}$. 
We now show that there exists a constant depending on the representation $\rho$ only.
\begin{prop}
\label{prop::varphi_PQ_final}
There exists a constant $B >0$ depending on the representation $\rho$ only such that for all connected components $P,Q, R \subset \tilde{S} \setminus \tilde{\lambda}$, for all $\varepsilon \in \Htwist$ with $\norm{\varepsilon} _{\mathcal{H}^{\mathrm{Twist}}(\widehat{\lambda};\mathfrak{a}_\theta)} < B$, the limit $\varphi^\varepsilon_{PQ} = \lim_{\mathcal{C} \to \mathcal{C}_{PQ}} \varphi^\varepsilon_{PQ}$ exists and satisfies the properties from Proposition \ref{prop::shearing_composition}.
\end{prop}
\begin{proof}
The only thing left to show is that the constant $B$ from Proposition \ref{prop::shearing_map} can be made independent of the arc $\tilde{k}$. 
Choose a collection $k_1, \dots, k_m$ of arcs on the surface $S$ transverse to $\lambda$ such that for every two connected components in $S \setminus \lambda$, there is an arc $k_j$ connecting them.
For every $j$, let $B_j >0$ be the constant as in Proposition \ref{prop::shearing_map} for a lift $\tilde{k}_{j}$ of $k_j$. 
This is independent of the choice of lift. 
Let $B := \min\{B_1, \dots, B_m\}$. 
Let $P, Q \subset \tilde{S} \setminus \tilde{\lambda}$ be connected components. 
Then there exists a finite sequence of components $P = R_0, R_1, \dots, R_N, R_{N+1} = Q$ such that $R_j$ separates $R_{j-1}$ from $R_{j+1}$ and such that  $R_j$ and  $R_{j+1}$ are connected by the same lift $\widetilde{k}_{i_j}$ of a transverse arc $k_{i_j}$. 
For $\norm{\varepsilon}_{\Htwist} <B$, the maps $\varphi_{R_j R_{j+1}}$ exist and thus also $\varphi_{PQ}  = \varphi_{PR_1} \cdots\varphi_{R_N Q}$.
\end{proof}

We conclude this subsection with an estimate that will be useful later.
\begin{lemma}
\label{lem::psi_P_estimate}
Let $r \in \mathbb{N}$ and let $P,Q \subset \tilde{S} \setminus \tilde{\lambda}$ be such that all $R \in \mathcal{C}_{PQ}$  satisfy  $r(R) >r$.
Then, for $\varepsilon \in \Htwist$ small enough, there exist constants $C', A' >0$ such that $\dG(\psi^\varepsilon_{PQ}, Id) \leq Ce^{-A'r}$.
\end{lemma}
\begin{proof}
As in the proof of Proposition \ref{prop::shearing_map}, we have
\begin{align*}
\dG(\psi^\varepsilon_{PQ}, \Id) &\leq  C_1 \sum_{R \in \mathcal{C}_{PQ}} \left( e^{C\norm{\varepsilon}(r(R)+1)} +1 \right) e^{-Ar(R)}.
\end{align*}
Using that all components in $\mathcal{C}_{PQ}$ have divergence radius at least $r+1$ and that the number of components with fixed divergence radius is bounded by some $D \in \mathbb{N}$ (Lemma \ref{lem::divergence_radius}), this gives us
\begin{align*}
\dG(\psi_{PQ}, \Id)
 &\leq C_1 D \sum_{l=r+1}^\infty \left( e^{C\norm{\varepsilon}(l+1)} \right) e^{-Al}  + C_1 \sum_{l=r+1}^\infty e^{-Al}.
\end{align*}
Both sums are the remainder term of the geometric series and for $\norm{\varepsilon} < A/C$, they are bounded by a constant times $e^{-A'r}$, where $A' := -(C \norm{\varepsilon} - A)$.
This finishes the proof.
\end{proof}

\subsection{Cataclysms}

From Proposition \ref{prop::shearing_map}, it follows that there exists a neighborhood $\mathcal{V}_\rho \subset \Htwist$ around $0$ depending on $\rho$ and a map
\begin{align*}
\mathcal{V}_\rho & \to G^{\{ (P,Q) \ | \ P,Q \subset \tilde{S} \setminus \tilde{\lambda} \} } \\
\varepsilon & \mapsto \{ \varphi^\varepsilon_{PQ} \}_{(P,Q)}
\end{align*}
that assigns to $\varepsilon$ the family of shearing maps.

Using the shearing maps, we can define the cataclysm deformation based at $\rho$ as follows:
Fix a reference component $P \subset \tilde{S} \setminus \tilde{\lambda}$ and a twisted cycle $\varepsilon \in \mathcal{V}_\rho$.
For $\gamma \in \pi_1(S)$, set
\begin{align*}
\Lambda_P^\varepsilon \rho (\gamma) := \varphi_{P (\gamma P)}^\varepsilon \rho(\gamma).
\end{align*}
With this definition, we have the following result.
\begin{theorem}
\label{thm::cataclysm_deformation_hom}
Let $\rho \colon \pi_1(S) \to G$ be a $\theta$-Anosov representation, $\lambda$ a geodesic lamination on $S$, not necessarily maximal, and $\tilde{\lambda}$ its lift to the universal cover $\tilde{S}$.
There exists a neighborhood $\mathcal{V}_\rho$ of $0$ in $\Htwist$ such that for any reference component $P \subset \tilde{S} \setminus \tilde{\lambda}$, there is a continuous map 
\begin{align*}
\Lambda_P \colon \mathcal{V}_\rho &\to \mathrm{Hom}(\pi_1(S), G) \\
\varepsilon &\mapsto \Lambda_P^\varepsilon \rho
\end{align*}
such that $\Lambda_P^0 \rho = \rho$.
This map is called \emph{cataclysm} based at $\rho$ along $\rho$.
Further, there exists a neighborhood $\mathcal{U}_\rho \subset \mathcal{V}_\rho$ such that for all $\varepsilon \in \mathcal{U}_\rho$, $\Lambda^\varepsilon_0 \rho$ is $\theta$-Anosov.
For different reference components $P, Q$, the resulting deformations $\Lambda_P^\varepsilon \rho$ and $\Lambda_Q^\varepsilon \rho$ differ by conjugation. 
\end{theorem}
\begin{proof}
The fact that $\Lambda_P^\varepsilon \rho$ is a group homomorphism follows from the $\rho$-equivariance and composition property of the shearing maps (Propositon \ref{prop::shearing_composition}).
The continuity results from the fact that the shearing maps depend continuously on the shearing parameter $\varepsilon$.
The existence of the neighborhood $\mathcal{U}_\rho$ follows from the fact that the set of $\theta$-Anosov representations is open in $\mathrm{Hom}(\pi_1(S), G)$ by \cite[Theorem 5.13]{GuichardWienhard_Anosov}.
Finally, the representations  $\Lambda_Q^\varepsilon \rho$ and $\Lambda_P^\varepsilon \rho$ are conjugate by $\varphi_{P Q}^\varepsilon$, which follows from a short computation, using the composition property of shearing maps. 
\end{proof}

\begin{remark}
Theorem \ref{thm::cataclysm_deformation_hom} is a generalization of a result in \cite{Dreyer_Cataclysms}, which covers the special case $G = \mathrm{PSL}(n, \mathbb{R})$.
Our result works in the much more general context of $\theta$-Anosov representations into a semisimple non-compact connected Lie group $G$ for $\theta \subset \Delta$ satisfying $\iota(\theta) = \theta$. 
Futher, in contrast to \cite{Dreyer_Cataclysms}, we do not assume that the lamination $\lambda$ is maximal.
\end{remark}

\begin{remark} 
Since conjugate representations $\rho$ and $\rho'$ give conjugate shearing maps, also the deformed representations $\Lambda_P^\varepsilon \rho$ and $\Lambda_P^\varepsilon \rho'$ are conjugate.
Thus, the map $\Lambda _P$ descends to a map on the character variety $\Lambda \colon \mathcal{U}_\rho \to \mathrm{Hom}_{\theta\mathrm{-Anosov}}(\pi_1(S), G)//G$ that is independent of the choice of reference component $P$.
\end{remark}

\subsection{The boundary map}

Anosov representations are often studied through their boundary map.
Thus, it is natural to ask how the boundary map changes under a cataclysm deformation.
We have the following result:

\begin{theorem}
\label{thm::boundary_map}
Let $\rho \colon \pi_1(S) \to G$ be $\theta$-Anosov, let $\varepsilon \in \mathcal{U}_\rho$ and let $\rho' := \Lambda^\varepsilon_P \rho$ be the $\varepsilon$-cataclysm deformation of $\rho$ along $\lambda$ with respect to a reference component $P \subset \tilde{S} \setminus \tilde{\lambda}$. 
Here,  $\mathcal{U}_\rho \subset \Htwist$ is as in Theorem \ref{thm::cataclysm_deformation_hom} such that $\rho'$ is $\theta$-Anosov.
Let $\zeta$ and $\zeta'$ be the boundary maps associated with $\rho$ and $\rho'$, respectively.
Then for every $x \in \partial_\infty \tilde{\lambda}$ that is a vertex of a connected component of $\tilde{S} \setminus \tilde{\lambda}$, the boundary map $\zeta^\varepsilon$ is given by
\begin{align}
\label{eq::zeta'_on_vertices}
\zeta'(x) = \varphi^\varepsilon_{P Q_x} \cdot \zeta(x).
\end{align}
\end{theorem}
\begin{proof}
A short computation shows that the right-hand side of \eqref{eq::zeta'_on_vertices} gives a well-defined and $\rho'$-equivariant boundary map.
The main observation for the proof is that a cataclysm deformation does not change the flag curve on the vertices of the reference component.
More precisely, if $x$ is a vertex of the reference component $P$, we have $\varphi_{P P}^\varepsilon \cdot \zeta(x)= \zeta(x) = \zeta'(x)$.
For the proof of this fact, we refer to \cite[Proposition 6.8]{Pfeil_Cataclysms}.
The idea is to first look at representations into $\SL{n}$, where we can use a linear algebraic estimate from \cite[Lemma A.6]{BPS}. 
For general $\theta$-Anosov representations into a Lie group $G$, we can use that every $\theta$-Anosov representation into a Lie group $G$ can be turned into a projective Anosov representation by composing it with a specific irreducible representation by \cite[Proposition 4.3 and Remark 4.12]{GuichardWienhard_Anosov}.

For an arbitrary point $x \in \partial_\infty \tilde{\lambda}$ that is a vertex of the component $Q_x \subset \tilde{S} \setminus \tilde{\lambda}$, we now use a trick and change the reference component:
If $x \in \partial_\infty \tilde{\lambda}$ is a vertex of the component $Q_x \subset \tilde{S} \setminus \tilde{\lambda}$, consider the $\varepsilon$-cataclysm deformation  $\rho'_x := \Lambda^\varepsilon_{Q_x} \rho$ with respect to the reference component $Q_x$.
Denote the corresponding boundary map by $\zeta'_x$. 
Then $\rho'$ is conjugated to $\rho'_x$ by $\varphi_{P Q_x}^\varepsilon$, so $\zeta' = \varphi_{P Q_x}^\varepsilon \cdot \zeta'_x$.
It follows that
\begin{align*}
\zeta'(x) = \varphi_{P Q_x}^\varepsilon \cdot \zeta'_x (x) = \varphi_{P Q_x}^\varepsilon \cdot \zeta (x),
\end{align*}
which finishes the proof.
\end{proof}

\section{Properties of cataclysms}

In addition to the lamination $\lambda$, the deformation depends on the representation $\rho$ we start with and the twisted cycle $\varepsilon$ that determines the amount of shearing. 
In this section, we ask how the deformation behaves when we change these parameters.
More precisely, we add two twisted cycles and we compose the representation $\rho$ with a group homomorphism. 

Let $\eta, \varepsilon \in \Htwist$ be small enough such that all the cataclysm deformations appearing in the following exist.
Intuitively, deforming first with $\varepsilon$ and then with $\eta$ should be the same as deforming with $\varepsilon + \eta$. 
Indeed, the following holds: 

\begin{theorem}
\label{thm::additivity}
The cataclysm deformation is additive, i.e.\ for a $\theta$-Anosov representation $\rho \colon \pi_1(S) \to G$ $\theta$-Anosov, a reference component $P \subset \tilde{S} \setminus \tilde{\lambda}$ and $\varepsilon, \eta \in \Htwist$ small enough, we have
\begin{align*}
 \Lambda_P^{\varepsilon + \eta} \rho = \Lambda_P^\varepsilon\left(\Lambda_P^\eta \rho \right).
\end{align*}
\end{theorem}
Here, we need to take into account where the deformations are \emph{based}, which is not reflected in the notation.
On the left, we look at the $(\varepsilon + \eta)$-cataclysm deformation based at $\rho$, on the right, at the $\varepsilon$-cataclysm deformation based at $\Lambda_P^\eta \rho$.  

To see why Theorem \ref{thm::additivity} holds, we need to understand how the shearing maps behave under adding transverse cycles.
Let $\rho' := \Lambda^\eta \rho$ be the deformed representation. 
We denote the shearing maps for $\rho'$ with ${\varphi'}_{PQ}^\varepsilon$.

\begin{prop}
\label{prop::additivity_shearing}
Let $P$ be the fixed reference component for the cataclysm deformation $\Lambda_P$.
Then for $\varepsilon, \eta \in \Htwist$ small enough, for every component $Q \subset \tilde{S} \setminus \tilde{\lambda}$ it holds that
\begin{align*}
{\varphi}_{P Q}^{\varepsilon + \eta} = {\varphi'}_{P Q}^{\varepsilon}   {\varphi}_{P Q}^{\eta} .  
\end{align*}
\end{prop}
\begin{proof}
Let ${T'}_g^H$ denote the stretching map for the deformed representation $\rho'$.
We first look at the stretching maps $T^H_g$ and ${T'}^H_g$ for an oriented geodesic $g$ and $H \in \atheta$. 
In Theorem \ref{thm::boundary_map} we proved that if $x$ is a vertex of the component $Q$, then that $\zeta'(x) = \varphi_{P Q}^{\eta} \! \cdot \!  \zeta(x)$.
Thus, if $g$ is an oriented geodesic bounding the component $Q$, the stretching maps satisfy 
\begin{align}
\label{eq::additivity_T'_conjugated}
{T'}_g^{H}  = \varphi_{PQ}^\eta {T}_g^{H} \left(\varphi_{PQ}^\eta \right)^{-1} ,
\end{align}
so ${T'}_g^{H}$ is conjugated to ${T}_g^{H}$ by $\varphi_{P Q}^\eta$.

We now consider a one-element subset $\{ R\} \subset \mathcal{C}_{PQ}$. 
Recall that $\psi_{P R}^\varepsilon = \varphi_{P R}^\varepsilon T_{g_R^0}^{-\varepsilon(P,R)}$.
Using \eqref{eq::additivity_T'_conjugated} and additivity of the cycle, we compute
\begin{align*}
{\varphi'}_{\{ R\}}^\varepsilon &= \left( {T'}_{g_R^0}^{\varepsilon(P,R)} {T'}_{g_R^1}^{-\varepsilon(P,R)} \right) {T'}_{g_Q^0}^{\varepsilon(P,Q)} \\
&= \left( \psi_{PR}^\eta  {T}_{g_R^0}^{\eta(P,R)} \right) 
\left( {T}_{g_R^0}^{\varepsilon(P,R)} {T}_{g_R^1}^{-\varepsilon(P,R)} \right)
\varphi_{RQ}^\eta {T}_{g_Q^0}^{\varepsilon(P,Q)} \left(\varphi_{PQ}^\eta \right)^{-1}  \\
&=   \psi_{PR}^\eta\left( {T}_{g_R^0}^{(\varepsilon + \eta) (P,R)} {T}_{g_R^1}^{-(\varepsilon+\eta)(P,R)} \right)
\left( {T}_{g_R^1}^{\eta(P,R)} \psi_{RQ}^\eta {T}_{g_Q^0}^{-\eta(P,R)} \right)
{T}_{g_Q^0}^{(\varepsilon+\eta)(P,Q)} \left(\varphi_{PQ}^\eta \right)^{-1}.
\end{align*}
Using the same techniques on a finite subset $\mathcal{C} = \{R_1, \dots, R_{m} \} \subset \mathcal{C}_{PQ}$ we find
\begin{align*}
{\varphi'}_{\mathcal{C}}^{\varepsilon} &= \prod_{i=1}^{m} \left( {T'}_{g_i^0}^{\varepsilon(P,R_i)} {T'}_{g_i^1}^{-\varepsilon(P,R_i)} \right) \quad {T'}_{g_Q^0}^{\varepsilon(P,Q)} \\
&=  \psi_{PR_1}^\eta \prod_{i=1}^{m} \left( \left( {T}_{g_i^0}^{(\varepsilon + \eta) (P,R_i)} {T}_{g_i^1}^{-(\varepsilon+\eta)(P,R_i)} \right)
\left( {T}_{g_i^1}^{\eta(P,R_i)} \psi_{R_i R_{i+1}}^\eta {T}_{g_{i+1}^0}^{-\eta(P,R_i)} \right) \right) \\ 
& \qquad \qquad {T}_{g_Q^0}^{(\varepsilon+\eta)(P,Q)} \left(\varphi_{PQ}^\eta \right)^{-1}.
\end{align*}
Set
\begin{align*}
\widehat{\psi}_{\mathcal{C}}^{\eta, \varepsilon} := {\varphi'}_{\mathcal{C}}^\varepsilon \left( {T}_{g_Q^0}^{(\varepsilon+\eta)(P,Q)} \left(\varphi_{PQ}^\eta \right)^{-1} \right)^{-1}.
\end{align*}
We claim that $\lim_{\mathcal{C} \to \mathcal{C}_{PQ} } \widehat{\psi}_{\mathcal{C}}^{\eta, \varepsilon} = \psi_{PQ}^{\varepsilon + \eta}$.
Recall that 
\begin{align*}
\psi_{PQ}^{\varepsilon+ \eta} = \lim_{\mathcal{C} \to \mathcal{C}_{PQ} } \psi_{\mathcal{C}}^{\varepsilon + \eta} = \lim_{m \to \infty} \prod_{i = 1}^{m} \left( {T}_{g_i^0}^{(\varepsilon + \eta) (P,R_i)} {T}_{g_i^1}^{-(\varepsilon+\eta)(P,R_i)} \right).
\end{align*}
This converges to $\psi_{PQ}^{\varepsilon + \eta}$ as $r$ goes to infinity.
Thus, it suffices to show that $\dG(\psi_{\mathcal{C}}^{\varepsilon + \eta} , \widehat{\psi}_{\mathcal{C}}^{\eta, \varepsilon} )$ tends to $0$ as $\mathcal{C}$ tends to $\mathcal{C}_{PQ}$.
The proof of this is technical, but does not use any new techniques or ideas.
Thus, we omit it here and refer the interested reader to \cite[Appendix A.2]{Pfeil_Cataclysms} for the details.

It follows that 
\begin{align*}
{\varphi'}_{PQ}^\varepsilon &= \lim_{\mathcal{C} \to \mathcal{C}_{PQ}} {\varphi'}_{\mathcal{C}}^{\varepsilon} \\
&= \lim_{\mathcal{C} \to \mathcal{C}_{PQ}} {\widehat{\psi}}_{\mathcal{C}}^{\eta, \varepsilon} T^{(\varepsilon + \eta)(P,Q)}_{g_Q^0} \left( \varphi_{PQ}^\eta \right)^{-1} \\
&= \psi_{PQ}^{\varepsilon + \eta} T^{(\varepsilon + \eta)(P,Q)}_{g_Q^0} \left( \varphi_{PQ}^\eta \right)^{-1} \\
&= \varphi_{PQ}^{\varepsilon + \eta}  \left( \varphi_{PQ}^\eta \right)^{-1},
\end{align*}
which finishes the proof.
\end{proof}

Theorem \ref{thm::additivity} is now a direct consequence of Proposition \ref{prop::additivity_shearing}. \\

Now we compose the $\theta$-Anosov representation $\rho \colon \pi_1(S) \to G$ with a homomorphism $\kappa \colon G \to G'$, where $G'$ is another semisimple connected non-compact Lie group with finite center.
We denote all objects associated with $G'$ with a prime. 
Denote by $\kappa_* \colon \mathfrak{a} \to \mathfrak{a}'$ the map induced by $\kappa$.
Let $W'_{\theta'}$ be the subgroup of the Weyl group $W'$ for $G'$ that fixes $\mathfrak{a}'_{\theta'}$ pointwise.
In \cite[Proposition 4.4]{GuichardWienhard_Anosov}, Guichard and Wienhard give sufficient conditions such that $\kappa \circ \rho$ is $\theta'$-Anosov for some $\theta' \subset \Delta'$:
\begin{prop}[{\cite[Proposition 4.4]{GuichardWienhard_Anosov}}]
\label{prop::Anosov_composition}
Let $\kappa \colon G \to G'$ be a Lie group homomorphism and assume that $\kappa(K) \subset K'$ and $\kappa(\mathfrak{a}) \subset \mathfrak{a}'$.
Let $\theta \subset \Delta$ and suppose that there exist $w' \in W'$ and $\theta' \subset \Delta'$ such that
\begin{align}
\label{eq::Anosov_composition_assumption}
\kappa_* \left( \overline{\mathfrak{a}^+} \setminus \bigcup_{\alpha \in \theta} \ker(\alpha) \right) \subset w' \cdot W'_{ \theta'} \cdot \left( \overline{{\mathfrak{a}'}^+} \setminus \bigcup_{\alpha \in  \theta'} \ker(\alpha') \right).
\end{align}
Then for any $\theta$-Anosov representation $\rho \colon \pi_1(S) \to G$, the representation $\kappa \circ \rho$ is ${\theta'}$-Anosov.
Furthermore $\kappa(P_\theta^\pm) \subset w'{P'}_{\theta'}^\pm {w'}^{-1}$, and there is an induced map $\kappa^+ \colon \mathcal{F}_\theta \to \mathcal{F}'_{\theta'}$.
If $\zeta \colon \partial_\infty \tilde{S} \to \mathcal{F}_\theta$ is the boundary map associated to $\rho$, then the boundary map for $\kappa \circ \rho$ is $\kappa^+ \circ \zeta$.
\end{prop}
The assumption \eqref{eq::Anosov_composition_assumption} guarantees that if $H \in \overline{\mathfrak{a}^+}$ does not lie in a wall of the Weyl chamber corresponding to an element  $\alpha \in \theta$, then also its image under $\kappa_*$ stays away from the walls corresponding to the elements $\alpha' \in \theta'$.
Note that the element $w'$ permuting the Weyl chambers cannot be omitted (see \cite[Example 5.24]{Pfeil_Cataclysms}).
\begin{remark}
The notation is used in \cite{GuichardWienhard_Anosov} is different from the notation we use -- what they call $\theta$-Anosov is $\left(\Delta \setminus \theta\right)$-Anosov in our notation.
Moreover, in their paper, there is a typing error in the statement and proof of their Proposition 4.4: In the assumption \eqref{eq::Anosov_composition_assumption},  $\theta$ and $\Delta \setminus \theta$ are reversed.
Here, we use our notational convention of $\theta$-Anosov and the  corrected assumption \eqref{eq::Anosov_composition_assumption}.
\end{remark}
The fact that $\kappa \circ \rho$ is again Anosov now triggers the question how such compositions of representations behave under cataclysm deformations. 
Indeed, under an additional assumption, we have:

\begin{prop}
\label{prop::cataclysm_equivariant}
Let $\kappa \colon G \to G'$ be a Lie group homomorphism and assume that $\kappa(K) \subset K'$ and $\kappa(\mathfrak{a}) \subset \mathfrak{a}'$.
Let $\theta \subset \Delta$ and $\theta' \subset \Delta'$ such that \eqref{eq::Anosov_composition_assumption} is satisfied. 
Further, assume that $\kappa_*(\atheta) \subset \mathfrak{a}'_{\theta'}$. 
Let $\rho \colon \pi_1(S) \to G$ be $\theta$-Anosov and let $\varepsilon \in \Htwist$ be sufficiently small such that $\Lambda^\varepsilon_0 \rho$ exists.
Then $\kappa \circ \rho$ is $\theta'$-Anosov and 
\begin{align}
\label{eq::cataclysm_equivariant}
\Lambda^{\kappa_*\varepsilon}_P \left( \kappa \circ \rho \right) = \kappa \left(\Lambda_P^\varepsilon \rho \right).
\end{align}
\end{prop}
Note that assumption $\kappa_*(\atheta) \subset \mathfrak{a}'_{\theta'}$ is not implied by \eqref{eq::Anosov_composition_assumption} as we will see in Example \ref{ex::composition_counterexample}.
Further, we want to remark that the notation $\Lambda_P$ for the cataclysm deformation does not encode the group $G$ containing the image of the representation $\rho$. 
This is given implicitly by the cycle.
In particular, in Proposition \ref{prop::cataclysm_equivariant}, $\Lambda^{\kappa_*\varepsilon}_P$ is the deformation of a representation with values in $G'$, and $\Lambda^{\varepsilon}_P$ is the deformation of a representation with values in $G$.

\begin{proof}[Proof of Proposition \ref{prop::cataclysm_equivariant}]
By Proposition \ref{prop::Anosov_composition}, the composition $\kappa \circ \rho$ is Anosov and the boundary map for $\kappa \circ \rho$ is given by $\kappa^+ \circ \zeta$, where $\kappa^+ \colon \mathcal{F}_\theta \to \mathcal{F}'_{\theta'}$ is a map induced by $\kappa$.
It follows that the stretching maps for $\kappa \circ \rho$ satisfy $T^{\kappa_* H}_g = \kappa \left( T^H_g \right)$ for any $H \in \atheta$ and oriented geodesic $g$. 
Here, $T^H_g$ is the stretching map for the representation $\rho$. 
Consequently, the shearing maps satisfy $\varphi_{PQ}^{\kappa_* \varepsilon} = \kappa \left( \varphi_{PQ}^\varepsilon \right)$, and the claim follows by the definition of $\Lambda_P$.
\end{proof}
We end this section with two examples - one family of representations for which the prerequisites of Proposition \ref{prop::cataclysm_equivariant} are satisfied, and one representation where it is not.

\begin{ex}
\label{ex::horocyclic}
An example where we can apply Proposition \ref{prop::cataclysm_equivariant} are \emph{$(n,k)$-horocyclic representations}.
They stabilize a $k$-dimensional subspace of $\mathbb{R}^n$ and are obtained from composing a discrete and faithful representation $\rho_0 \colon \pi_1(S) \to \SL{2}$ with the reducible representation $\iota_{n,k} \colon \SL{2} \to \SL{k}$ defined by
\begin{align*}
\iota_{n,k} \colon \SL{2} &\to \SL{n}, \\
\begin{pmatrix}
a & b \\ c & d 
\end{pmatrix} 
&\mapsto
\begin{pmatrix}
a \Id_k & 0 & b \Id_k \\ 0 & \Id_{n-2k} & 0 \\ c \Id_k & 0 & d \Id_k
\end{pmatrix}.
\end{align*}
One can check directly that $\iota_{n,k}$ satisfies all prerequisites of Proposition \ref{prop::cataclysm_equivariant}, so we have $
\Lambda^{{\iota_{n,k}}_*\varepsilon}_0 \left( \iota_{n,k} \circ \rho_0 \right) = \iota_{n,k} \left(\Lambda_P^\varepsilon \rho_0 \right)$.
In particular, if we deform an $(n,k)$-horocyclic representation using cycles in $\left(\iota_{n,k}\right)_* \left( \Htwist \right)$, then the resulting representation is again $(n,k)$-horocyclic.
We will look more closely at deformations of $(n,k)$-horocyclic representations in Section \ref{sec::horocyclic}.

\end{ex}
\begin{ex}
\label{ex::composition_counterexample}
We now give an example where $\kappa \circ \rho$ is an Anosov representation, but the additional assumption $\kappa_*(\atheta) \subset \mathfrak{a}'_{\theta'}$ in Proposition \ref{prop::cataclysm_equivariant} is not satisfied.
Let $G= \SL{5}$, $\theta = \{ \alpha_2, \alpha_3\}$ and let $\kappa := \bigwedge^2_5 \colon \SL{5} \to \SL{10}$ be the exterior power representation.
Let $\theta' := \{ \alpha_1, \alpha_9\}$.
Up to applying an element in $W'_{\theta'}$,  $\left( \bigwedge^2_5 \right)_* (H)$ lies in $\overline{{\mathfrak{a}'}^+}$  for all  $H \in \overline{\mathfrak{a}^+} \setminus \bigcap_{\alpha \in \theta} \ker (\alpha)$, so \eqref{eq::Anosov_composition_assumption} is satisfied and by \cite[Propositon 4.4]{GuichardWienhard_Anosov}, $\kappa \circ \rho$ is $\theta'$-Anosov. 
However, it is easy to check that the additional assumption $\left( \bigwedge^2_5\right)_*(\atheta) \subset \mathfrak{a}'_{\theta'}$ from Proposition \ref{prop::cataclysm_equivariant} is not satisfied.
This example shows that Proposition \ref{prop::additivity_shearing} does not apply in general when we compose an Anosov representation with an irreducible representation into $\SL{n}$ in order to obtain a projective Anosov representation.
\end{ex}

\section{Injectivity properties of cataclysms}

For Teichm\"uller space, cataclysms are injective and give rise to shearing coordinates (see \cite{Bonahon_Shearing}). 
This motivates the question if they are also injective in the general case.
For Hitchin representations into $\mathrm{PSL}(n, \mathbb{R})$, this is true (see Corollary \ref{cor::injective_SLn}).
However, for general $\theta$-Anosov representations, it does not hold.
In this section, we show that for $(n,k)$-horocyclic representations, the cataclysm deformation is not injective.
We then present a sufficient condition for injectivity.

\subsection{Horocyclic representations}
\label{sec::horocyclic}

Let $\rho$ be an $(n,k)$-horocyclic representation as introduced in Example \ref{ex::horocyclic}, obtained from composing a Fuchsian representation into $\SL{2}$ with the reducible representation $\iota_{n,k}$ whose image stabilizes a $k$-plane.
A horocyclic representation is reducible and has non-trivial centralizer, which will play an important role in the proof of the following result.

\begin{theorem}
\label{thm::Htrivial}
Let $\rho \colon \pi_1(S) \to \SL{n}$ be $(n,k)$-horocyclic.
Then there is a subspace $\mathcal{H}_{\mathrm{trivial}} \subset \mathcal{H}(\widehat{\lambda}; \atheta)$ such that $\Lambda^\varepsilon_0 \rho = \rho$ if and only if $\varepsilon \in \mathcal{H}_{\mathrm{trivial}}$.
The subspace $\mathcal{H}_{\mathrm{trivial}} $ depends only on the lamination $\lambda$.
The dimension of $\mathcal{H}_{\mathrm{trivial}}$ can be estimated by
\begin{align}
\label{eq::dimension_kerf_estimate}
-\chi(\lambda) + n(\lambda) - 2g(S) \leq \dim \mathcal{H}_{\mathrm{trivial}} \leq -\chi(\lambda) + n(\lambda),,
\end{align}
where $\chi(\lambda)$ is the Euler characteristic of $\lambda$ and $n(\lambda)$ the number of connected components.
\end{theorem}
\begin{proof}
We can split up $\atheta$ as $\atheta = \left(\iota_{n,k}\right)_*(\mathfrak{a}_2) \oplus \mathfrak{a}'$, where $\mathfrak{a}_2$ is the maximal abelian subalgebra for $\SL{2}$, $\left(\iota_{n,k}\right)_* \colon \mathfrak{a}_2 \to \mathfrak{a}$ is induced from $\iota_{n,k}$ and $\mathfrak{a}'$ is the $1$-dimensional subspace given by
\begin{align*}
\mathfrak{a}' :=  \left\{ \begin{pmatrix}
a  \Id_k& & \\ & -2a \Id_{n-2k} & \\ & & a \Id_k
\end{pmatrix} \middle| a \in \mathbb{R} \right\} \subset \mathfrak{a}.
\end{align*}
By Example \ref{ex::horocyclic}, we know that for cycles with values in $\left(\iota_{n,k}\right)_*\left( \mathfrak{a}_2 \right)$, the cataclysm deformation based at $\rho$ is injective, since it is the image under $\iota_{n,k}$ of a cataclysm deformation of the representation $\rho_0$ in $\SL{2}$. 
These are injective by Corollary \ref{cor::injective_SLn} below.
For cycles with values in $\mathfrak{a}'$ we observe that $\exp(\mathfrak{a}')$ lies in the centralizer of $\iota_{n,k}(\SL{2})$. 
It follows that the stretching maps satisfy $T_g^H = \exp(H)$ for all $H \in \mathfrak{a}'$ and all oriented geodesics $g$.
By additivity of the cycle, this implies that $\varphi_{PQ}^\varepsilon = \exp\left(\varepsilon(P,Q) \right)$ for all connected components $P$ and $Q$ and all $\mathfrak{a}'$-valued transverse twisted cycles $\varepsilon \in \mathcal{H}^\mathrm{Twist}\left(\olambda; \mathfrak{a}'\right)$.
By definition of the cataclysm deformation, we thus have
\begin{align}
\label{eq::horocyclic_trivial_iff}
\Lambda_P^\varepsilon \rho = \rho \qquad \mathrm{if \ and \ only \ if \ } \qquad \varepsilon(P, \gamma P) = 0 \ \forall \gamma \in \pi_1(S).
\end{align}
Since $\mathfrak{a}'$ is one-dimensional, we can identify $\mathcal{H}\left( \olambda; \mathfrak{a}'\right)$ with $\mathbb{R}$-valued transverse cycles $\mathcal{H}\left( \olambda; \mathbb{R}\right)$.
Restricting to twisted cycles, there is a one-to-one correspondence between $\mathcal{H}^{\mathrm{Twist}}\left( \olambda; \mathfrak{a}'\right)$ and $\mathcal{H}\left( \olambda; \mathbb{R}\right)^-$, the $(-1)$-eigenspace of $\mathcal{H}\left( \olambda; \mathbb{R}\right)$ with respect to the orientation reversing involution $\mathfrak{R}$.
Define 
\begin{align*}
f \colon \mathcal{H}(\widehat{\lambda}; \mathbb{R})^- &\to \mathrm{Hom}(\pi_1(S), \mathbb{R}), \\
\varepsilon &\mapsto \left( u_{\varepsilon} \colon \gamma \mapsto \varepsilon(P, \gamma P) \right).
\end{align*}
A short computation shows that this is a group homomorphism and independent of the reference component.
Together with \eqref{eq::horocyclic_trivial_iff}, we now have that $\Lambda_P^\varepsilon \rho = \rho$ if and only if $\varepsilon$ takes values in $\mathfrak{a}'$ and lies in $\ker(f)=: \mathcal{H}_{\mathrm{trivial}}$. 
Here, we also use the additivity of the cataclysm deformation (Theorem \ref{thm::additivity}) and that for representations into $\SL{2}$, the cataclysm deformation is injective (Corollary \ref{cor::injective_SLn}).
By the dimension formula, the dimension of $\ker(f)$ lies between $\mathrm{dim} \left( \mathcal{H}\left( \olambda; \mathbb{R}\right)^- \right)  = -\chi(\lambda) + n(\lambda)$ and $\mathrm{dim} \left( \mathcal{H}\left( \olambda; \mathbb{R}\right)^- \right) - \mathrm{dim} \left( \mathrm{Hom}(\pi_1(S), \mathbb{R}) \right) =  -\chi(\lambda) + n(\lambda) - 2g$, which finishes the proof.
A more detailed version of the proof can be found in \cite[Section 7.4]{Pfeil_Cataclysms}.
\end{proof}
\begin{remark}
For $(3,1)$-horocyclic representations, cataclysm deformations with cycles in $ \mathcal{H}^{\mathrm{Twist}}(\widehat{\lambda}; \mathfrak{a}')$ agree with \emph{$u$-deformations}, that have been introduced by Barbot in \cite[Section 4.1]{Barbot_3dimAnosov} for $u \in \mathrm{Hom}(\pi_1(S), \mathbb{R})$.
Barbot also gives a precise condition  on $u$ for which the deformation $\rho_u$ is Anosov \cite[Theorem 4.2]{Barbot_3dimAnosov}.
Deformations of this form also appear as \emph{bulging deformations} of convex projective structures in \cite{Goldman_Bulging} and \cite{Wienhard_Zhang_Deforming_CRPS}.
\end{remark}

\begin{cor}
If the lamination $\lambda$ is maximal, then the subspace $\mathcal{H}_{\mathrm{trivial}}$ in Theorem \ref{thm::Htrivial} has dimension $4g(S)-5$.
\end{cor}
\begin{proof}
For a maximal lamination, the complement $S \setminus \lambda$ consists of $4g-4$ connected components that are ideal triangles.
Let $\mathcal{C}'_{\lambda} := \{ Q_0, Q_1, \dots, Q_{4g-5}\}$ be a set of representatives of connected components of $\tilde{S} \setminus \tilde{\lambda}$ such that each component of $S \setminus \lambda$ has a lift contained in $\mathcal{C}'_\lambda$.
Let $P = Q_0$ be the fixed reference component, and let $\mathcal{C}_\lambda := \mathcal{C}'_\lambda \setminus \{ P \}$.
Further, let $\alpha_i, \beta_i$ for $i= 1, \dots, g(S)$ be generators of $\pi_1(S)$.
Define the map
\begin{align*}
V \colon \mathcal{H}(\widehat{\lambda}; \mathbb{R})^- &\to \mathbb{R}^{2g} \times \mathbb{R}^{\mathcal{C}_\lambda}, \\
\varepsilon &\mapsto \left( \left(\varepsilon(P, \alpha_i P), \varepsilon(P, \beta_i P) \right)_{i = 1, \dots, g}, \left( \varepsilon(P,Q) \right)_{Q \in \mathcal{C}_{\lambda}} \right).
\end{align*}
It is straightforward to check that $V$ is a vector space homomorphism.
From additivity of the transverse cycle $\varepsilon$, it follows that $\varepsilon$ is uniquely determined by $V(\varepsilon)$, so $V$ is injective. 
For dimension reasons, $V$ is also surjective, so a bijection.
By construction, $\ker(f)$ agrees with $V^{-1}\left( \{ 0\}^{2g(S)} \times \mathbb{R}^{\mathcal{C}_\lambda} \right)$, which has dimension $4g(S) - 5$.
\end{proof}

\begin{remark}
An analogous construction as in Theorem \ref{thm::Htrivial} works for reducible $\Delta$-Anosov representations into $\SL{2n+1}$ that are obtained from composing a $\Delta_{\Sp{2n}}$-Anosov representation $\rho_0 \colon \pi_1(S) \to \Sp{2n}$ with a reducible embedding $\iota_{2n \to 2n+1} \colon \Sp{2n} \to \SL{2n+1}$. 
With the same arguments as for the case of $(n,k)$-horocyclic representations, it holds that there exists a subspace $\mathcal{H}_{\mathrm{trivial}} \subset \mathcal{H}^{\mathrm{Twist}}(\widehat{\lambda}; \mathfrak{a})$ such that $\Lambda^\varepsilon_P \rho = \rho$ for all $\varepsilon \in \mathcal{H}_{\mathrm{trivial}}$.
In this case, we can split up $\mathfrak{a}$ as ${\left(\iota_{2n \to 2n+1}\right)}_{*}(\mathfrak{a}_{2n}) \oplus \mathfrak{a}' \oplus \mathfrak{a}''$, where $\mathfrak{a}'$ is $1$-dimensional and $\mathcal{H}_{\mathrm{trivial}} \subset \mathcal{H}^{\mathrm{Twist}}(\widehat{\lambda}; \mathfrak{a}')$ identifies with $\mathcal{H}\left( \olambda; \mathbb{R}\right)^-$. 
Since we cannot determine the behavior of $\Lambda_P$ for cycles with values in $\mathfrak{a}''$, we only obtain a  lower bound on the dimension of $\mathcal{H}_{\mathrm{trivial}}$ in contrast to Theorem  \ref{thm::Htrivial}.
\end{remark}

\begin{remark}
Horocyclic representations are reducible. 
However, we can also construct irreducible representations for which injectivity of the cataclysm deformation fails.
An example are \emph{hybrid representations}.
These were constructed for representations into $\Sp{4}$ in \cite[§3.3.1]{GuichardWienhard_TopologicalInvariants}.
We use the same technique for $\SL{3}$ to obtain an irreducible representation whose restriction to a subsurface is reducible:
Let $S$ be a closed connected oriented surface of genus at least $2$ and let $c$ be a simple closed separating curve.
Let $S_l$ and $S_r$ be the connected components of $S \setminus c$.
Then the fundamental group of $S$ is the amalgamated product $\pi_1(S)= \pi_1(S_l) *_{c} \pi_1(S_r)$. 
Let $\rho_0 \colon \pi_1(S) \to \SL{2}$ be discrete and faithful.
We can assume that $\rho_0(c)$ is diagonal with eigenvalues $e^{a}$ and $e^{-a}$.
For $t \in [0,1]$ let $\rho_t \colon \pi_1(S) \to \SL{2}$ be a continuous path of representations starting at $\rho_0$ such that $\rho_1(c)$ is diagonal with entries $e^{2a}$ and $e^{-2a}$. 
Set $\rho_{l} := \iota_{3,1} \circ \rho_1$, where $\iota_{3,1} \colon \SL{2} \to \SL{3}$ is the reducible representation introduced in Example \ref{ex::horocyclic}.
Further, let $\rho_r := j_3 \circ \rho_0$, where $j_3 \colon \SL{2} \to \SL{3}$ is the unique irreducible representation. 
Then $\rho_l(c) = \rho_r(c)$,
so we can define $\rho := \rho_l|_{\pi_1(S_l)} *_{c} \rho_r|_{\pi_1(S_r)}$.
Since $\rho|_{\pi_1(S_r)}$ is irreducible also $\rho$ is irreducible. 
Let $\lambda$ be a finite lamination that is subordinate to a pair of pants decomposition containing the separating curve $c$.
In  \cite[Subsection 7.3]{Pfeil_Cataclysms}, we construct a transverse cycle $\varepsilon$ that it is trivial on all arcs that are contained in $S_r$, i.e.\ on the part of the surface where $\rho$ is irreducible, and 
such that $\varphi^\varepsilon_{P (\gamma P)} = \Id$ for all $\gamma \in \pi_1(S)$.
In particular, the $\varepsilon$-cataclysm deformation along $\lambda$ is trivial. 
\end{remark}

\subsection{Injectivity of shearing maps and a sufficient condition for injectivity}

Even though the cataclysm deformation is not injective in general, we can give a sufficient condition for injectivity. 
To do so, we first observe that the map that assigns to a twisted cycle the family of shearing maps, is injective.

\begin{prop}
\label{prop::shearing_injective}
If two transverse twisted cycles $\varepsilon, \eta \in \Htwist$ have the same family of shearing maps, i.e. $\varphi^\varepsilon_{PQ} = \varphi^\eta_{PQ}$ for all $P,Q \subset \tilde{S} \setminus \tilde{\lambda}$, then $\varepsilon = \eta$.
In other words, the map 
\begin{align*}
\mathcal{V}_\rho &\to G^{ \{(P,Q) | P,Q \subset \tilde{S} \setminus \tilde{\lambda} \} } \\
\varepsilon &\mapsto \{\varphi^\varepsilon_{PQ} \}_{(P,Q)}
\end{align*}
is injective. 
Here, $\mathcal{V}_\rho \subset \Htwist$ is the neighborhood of $0$ from Proposition \ref{prop::shearing_map}.
\end{prop}

For the proof of Proposition \ref{prop::shearing_injective}, we make use of the \emph{Busemann cocycle} $\sigma \colon G \times \mathcal{F}_\theta \to \atheta$ (see  \cite[Lemma 6.29]{BenoistQuint_RandomWalks}).
It is defined as follows: 
Let $G = KAN^+$ be the Iwasawa decomposition. 
The maximal compact subgroup $K$ acts transitively on $\mathcal{F}_\Delta = G/B^+$. 
Thus, every element in $\mathcal{F}_\Delta$ can be written as $k\cdot B^+$ for some $k \in K$.
By the Iwasawa decomposition, for $g \in G$ and $F = k \cdot B^+ \in \mathcal{F}_\Delta$, there exists a unique element $\sigma(g, F)$ in $\mathfrak{a}$ such that 
\begin{align*}
gk \in K \exp(\sigma(g, F)) N^+.
\end{align*} 
The Busemann cocycle also exists in the context of partial flag manifolds: 
Let $W_\theta$ be the subgroup of the Weyl group $W$ that fixes $\atheta$ point-wise and let $p_\theta \colon \mathfrak{a} \to \atheta$ be the unique projection invariant under $W_\theta$.
For every $\theta \subset \Delta$, the map $p_\theta \circ \sigma \colon G \times \mathcal{F}_\Delta \to \atheta$ factors through a map $\sigma_{\theta} \colon G \times \mathcal{F}_\theta \to \atheta$ (by \cite[Lemma 8.21]{BenoistQuint_RandomWalks}).
The Busemann cocycle is continuous and satisfies the \emph{cocycle property}, i.e.\ for all $g_1, g_2 \in G$ and $F \in \mathcal{F}_\theta$, 
\begin{align*}
\Buse{g_1g_2}{F} = \Buse{g_1}{g_2 \cdot F} + \Buse{g_2}{F}. 
\end{align*}
By definition, the Busemann cocycle satisfies
\begin{align*}
\Buse{\mathrm{id}}{F} = 0 \ \forall F \in \mathcal{F}_\Delta \qquad \mathrm{and} \qquad \Buse{\exp(H)}{P_\theta^+} = H \  \forall H \in \atheta.
\end{align*}
\begin{proof}[Proof of Proposition \ref{prop::shearing_injective}]
We use the Busemann cocycle $\sigma_\theta$ and define for an arc $\tilde{k}$ transverse to $\tilde{\lambda}$
\begin{align*}
\delta(\tilde{k}) & := \sum_{R \in \mathcal{C}_{PQ}} \bigg(\Buse{\varphi^\varepsilon_R}{P_{g_R^0}^+}- \Buse{\varphi^\varepsilon_R}{P_{g_R^1}^+}\bigg)
- \Buse{\Id}{P_{g_P^1}^+} + \Buse{\varphi^\varepsilon_Q}{P_{g_Q^0}^+},
\end{align*}
where for a component $R \subset \tilde{S} \setminus \tilde{\lambda}$, we set $\varphi^\varepsilon_{R} := \varphi^\varepsilon_{PR}$.
We first show that the sum defining $\delta$ converges.

In the following, we omit the superscript $\varepsilon$.
Fix a norm $\norm{\cdot}$ on $\mathfrak{a}$.
The Busemann cocycle is analytic, so in particular, it is locally Lipschitz.
Since the arc $\tilde{k}$ is compact, all the shearing maps $\varphi_R$ lie within a compact subset of $G$, and the flags $P_{g_R^{0/1}}^+$ lie within a compact subset of $\mathcal{F}_\theta$.
Thus, using H\"older continuity of the boundary map, Lemma \ref{lem::divergence_radius} and the fact that $\mathrm{d}(g_R^0, g_R^1) \leq B \ell(\tilde{k} \cap R)$ for some constant $B>0$, there exist constants $C_i, A_i > 0$ depending on the arc $\tilde{k}$ and the representation $\rho$ such that
 \begin{align*}
 \norm{\Buse{\varphi_R}{P_{g_R^0}^+}- \Buse{\varphi_R}{P_{g_R^1}^+} } &\leq C_1 \  \mathrm{d}_{\mathcal{F}_\theta}\left(P_{g_R^0}^+, P_{g_R^1}^+ \right)  \\
 &\leq C_2 \  \mathrm{d}\left( g_R^0, g_R^1 \right)^{A_1} \\
& \leq C_3 e^{-A_2 r(R)}.
 \end{align*}
It follows that
\begin{small}
\begin{align*}
\sum_{R \in \mathcal{C}_{PQ}} \norm{\Buse{\varphi_R}{P_{g_R^0}^+}- \Buse{\varphi_R}{P_{g_R^1}^+}} 
\leq C_3 \sum_{R \in \mathcal{C}_{PQ}} e^{-A_2r(R)}
 \leq C_4 \sum_{r=0}^{\infty}e^{-A_2 r} 
 < \infty,
\end{align*}
\end{small}
where for the last inequality, we use the fact that the number of all components with fixed divergence radius is uniformly bounded (Lemma \ref{lem::divergence_radius}).
In total, the sum defining $\delta$ is absolutely convergent.

We now show that $\delta(\tilde{k}) = \varepsilon(P,Q)$.
First, we observe that if $g$ is an oriented geodesic in $\tilde{S}$ and $H \in \atheta$, then 
\begin{align}
\label{eq::Busemann_basic}
\Buse{T^H_g}{P_g^+} = H.
\end{align}
This follows from the definition of $T_g^H$ and the cocycle property.
As in the proof of Proposition \ref{prop::shearing_map}, we fix a sequence $\left(\mathcal{C}_m \right)_{m \in \mathbb{N}}$ of subsets of $\mathcal{C}_{PQ}$ such that $\mathcal{C}_m$ has cardinality $m$ and $\mathcal{C}_m \subset \mathcal{C}_{m+1}$ for all $m$.
For a fixed $m$, let $\mathcal{C}_m= \{R_1, R_2, \dots, R_m\}$, where the labelling is from $P$ to $Q$. 
Set $R_0 := P$ and $R_{m+1} := Q$.
Note that, to be precise, we would have to record the subset $C_m$ in the notation of the $R_k$, i.e.\ $\mathcal{C}_m = \{ R^m_1, \dots R^m_m\}$. 
For instance, $R^m_k$ is either equal to  $R^{m+1}_k$ or to $R^{m+1}_{k+1}$, depending on whether the unique component $R \in \mathcal{C}_{m+1}\setminus \mathcal{C}_m$ separates $R^m_k$ from $P$ or from $Q$.
However, we feel that this would overload the notation even more, so we omit the superscript, keeping in mind that the components $R_i$ depend on the on the subset $\mathcal{C}_m$. 
To shorten notation, let  $\varphi_i := \varphi_{R_i}$ and $g_i^{0/1} := g_{R_i}^{0/1}$.
Reordering the sum defining $\delta$, we have 
\begin{align*}
\delta(\tilde{k}) &= \lim_{m \to \infty}\sum_{R_i \in \mathcal{C}_m} \bigg(\Buse{\varphi_i}{P_{g_i^0}^+}- \Buse{\varphi_i}{P_{g_i^1}^+}\bigg)
- \Buse{\Id}{P_{g_0^1}^+} + \Buse{\varphi_{m+1}}{P_{g_{m+1}^0}^+}\\
&= \lim_{m \to \infty} \sum_{i=0}^{m} \bigg(\Buse{\varphi_{i+1}}{P_{g_{i+1}^0}^+}- \Buse{\varphi_i}{P_{g_i^1}^+}\bigg).
\end{align*}
We can write the shearing map as a composition $\varphi_{i+1} = \varphi_i \psi_{i} T^{\varepsilon(R_i, R_{i+1})}_{g_{i+1}^0}$, where $\psi_{i} := \psi_{R_i R_{i+1}}$. 
Using the cocycle property together with the fact that  $T^{\varepsilon(R_i, R_{i+1})}_{g_{i+1}^0}$ stabilizes $P_{g_{i+1}^0}^+$, and the observation in \eqref{eq::Busemann_basic}, we obtain
\begin{align*}
\Buse{\varphi_{i+1}}{P_{g_{i+1}^0}^+} 
&= \Buse{\varphi_i   \psi_{i}} {P_{g_{i+1}^0}^+} + \varepsilon(R_i, R_{i+1}).
\end{align*}
Thus, by additivity of $\varepsilon$ and the cocycle property, 
\begin{align*}
\delta(\tilde{k}) &= \lim_{m \to \infty} \sum_{i=0}^{m} \bigg(  \Buse{\varphi_i   \psi_{i}} {P_{g_{i+1}^0}^+} +  \varepsilon(R_i, R_{i+1}) - \Buse{\varphi_i}{P_{g_i^1}^+}  \bigg) \\
&= \varepsilon(P,Q) +  \lim_{m \to \infty}  \sum_{i=0}^{m} \bigg(  \Buse{\varphi_i}{\psi_{i}   P_{g_{i+1}^0}^+} - \Buse{\varphi_i}{P_{g_i^1}^+} + \Buse{\psi_{i}} {P_{g_{i+1}^0}^+}   \bigg)
\end{align*}
It remains to show that the limit on the right side equals zero.
Fix $m \in \mathbb{N}$.
By local Lipschitz continuity of the Busemann cocycle, there exists a constant $C_1 >0$ depending on $\tilde{k}$ and $\rho$ such that
\begin{align*}
\norm{\Buse{\varphi_i}{\psi_{i}  P_{g_{i+1}^0}^+} - \Buse{\varphi_i}{P_{g_i^1}^+}}
&\quad \leq  C_1 \ \mathrm{d}_{\mathcal{F}_\theta}\left( \psi_i  P_{g_{i+1}^0}^+, P_{g_i^1}^+\right) \\
&\quad \leq C_1 \left( \mathrm{d}_{\mathcal{F}_\theta}\left( \psi_i  P_{g_{i+1}^0}^+, P_{g_{i+1}^0}^+\right) + \mathrm{d}_{\mathcal{F}_\theta}\left(P_{g_{i+1}^0}^+, P_{g_i^1}^+\right) \right).
\end{align*}
Let $r_m := \min_{R \in \mathcal{C}_{PQ} \setminus \mathcal{C}_m} r(R)$ be the minimal divergence radius of all components of $\mathcal{C}_{PQ}$ that are not contained in $\mathcal{C}_m$.
It goes to infinity as $m$ tends to $\infty$, since for fixed  $n$, there are only finitely many components in $\mathcal{C}_{PQ}$ with $r(R) = n$.
By H\"older continuity of the boundary map, there are constants $A_j, C_j >0$ depending on $\tilde{k}$ and $\rho$ such that
\begin{align*}
\mathrm{d}_{\mathcal{F}_\theta}\left(P_{g_{i+1}^0}^+, P_{g_i^1}^+\right)  &\leq C_2\mathrm{d}\left(g_{i+1}^0, g_i^1\right)^{A_1} \\
&\leq  C_3 \left(\sum_{R \in \mathcal{C}_{g^0_{i+1} g_i^1}} e^{-A_2 r(R)} \right)^{A_1} \\
& \leq C_4 e^{-A_3 r_m},
\end{align*}
where for the last step, we used that the series can be estimated by the remainder term of a geometric series and is bounded by a constant times $e^{-A_3 r_m}$.

Further, since the action of $G$ on $\mathcal{F}_\theta$ is smooth, it is in particular locally Lipschitz and we have by Lemma \ref{lem::psi_P_estimate}
\begin{align*}
\mathrm{d}_{\mathcal{F}_\theta}\left( \psi_i   P_{g_{i+1}^0}^+ , P_{g_{i+1}^0}^+\right) \leq C_6 \dG( \psi_i, \Id) \leq C_7 e^{-A_4r_m}.
\end{align*}
Combining these estimates gives us 
\begin{align}
\label{eq::delta_equals_varepsilon_1}
\norm{\Buse{\varphi_i}{\psi_{i}   P_{g_{i+1}^0}^+} - \Buse{\varphi_i}{P_{g_i^1}^+}} \leq C_8 e^{-A_5r_m}.
\end{align}
In addition, again by Lemma \ref{lem::psi_P_estimate}, we have
\begin{align}
\label{eq::delta_equals_varepsilon_2}
\norm{ \Buse{\psi_{i}} {P_{g_{i+1}^0}^+}} 
&= \norm{\Buse{\psi_{i}} {P_{g_{i+1}^0}^+} - \Buse{\Id} {P_{g_{i+1}^0}^+}}   \\
&\leq C_9 \dG(\psi_{i},\Id) \nonumber \\
&\leq C_{10} e^{-A_6r_m}.  \nonumber
\end{align}
Combining the estimates from \eqref{eq::delta_equals_varepsilon_1} and \eqref{eq::delta_equals_varepsilon_2}, we have
\begin{align*}
 \sum_{i=0}^{m} \norm{\Buse{\varphi_i   \psi_{i}} {P_{g_{i+1}^0}^+} - \Buse{\varphi_i}{P_{g_i^1}^+}} 
\leq \sum_{i=0}^{m} C_{11}e^{-A_7 r_m}
\leq C (r_m+1) e^{-A r_m},
\end{align*}
where we use the fact that $m = |\mathcal{C}_m|$ is bounded by a constant times $r_m+1$ (Lemma \ref{lem::divergence_radius}). 
If $m$ goes to infinity, $r_m$ goes to infinity, so the right hand side converges to zero. 
This finishes the proof that  $\delta(\tilde{k}) = \varepsilon(P,Q)$.
Since $\delta$ depends on the family of shearing maps, and only indirectly on $\varepsilon$, it follows that $\varepsilon = \eta$ if they have the same families of shearing maps.
\end{proof}
\begin{remark}
In \cite{Dreyer_Cataclysms}, Dreyer shows Proposition \ref{prop::shearing_injective} for the case of $\Delta$-Anosov representations into $\mathrm{PSL}(n, \mathbb{R})$.
He does not use the Busemann cocycle, but the viewpoint on Anosov representations through a bundle over $T^1 S$ and considers a flow on this bundle. 
Our approach gives the same results in the special case of $\Delta$-Anosov representations into $\mathrm{PSL}(n, \mathbb{R})$, but works in the more general setting.
Further, Dreyer concludes from this result injectivity of the cataclysm deformation, which is wrong as we have seen in Theorem \ref{thm::Htrivial} above.
\end{remark}

With Proposition \ref{prop::shearing_injective} and an assumption on the flag curve of $\rho$, we can guarantee injectivity of the cataclysm deformation.
\begin{cor}
\label{cor::stab_trivial_injective}
Let $\rho \colon \pi_1(S) \to G$ be $\theta$-Anosov with flag curve $\zeta \colon \partial_\infty \tilde{S}$ such that for every connected component $Q \subset \tilde{S} \setminus \tilde{\lambda}$, 
\begin{align*}
\bigcap_{x \in \partial Q} \mathrm{Stab}_G \zeta(x) = \{ Id \}.
\end{align*}
Then for any fixed reference triangle $P$, the cataclysm deformation based at $\rho$ is injective.
\end{cor}
\begin{proof}
Let $\varepsilon, \eta \in \mathcal{U}_\rho \subset \Htwist$ be such that $\Lambda_P^\varepsilon \rho = \Lambda_P^\eta \rho$ and let $\zeta_\varepsilon = \zeta_\eta$ be the corresponding boundary curve. 
By Theorem \ref{thm::boundary_map}, we have for every $x$ that is a vertex of the connected component $Q_x \subset \tilde{S} \setminus \tilde{\lambda}$,
\begin{align*}
\varphi_{P Q_x}^\varepsilon \cdot \zeta(x) = \zeta_\varepsilon(x) =\zeta_\eta(x) = \varphi_{P Q_x}^\eta \cdot \zeta(x),
\end{align*}
so $\varphi_{P Q_x}^\varepsilon \left( \varphi_{P Q_x}^\eta \right)^{-1}$ lies in the stabilizer of $\zeta(x)$. 
This holds for any $x$ in the boundary of $Q_x$, so by the assumption on $\zeta$, we have $\varphi_{P Q_x}^\varepsilon = \varphi_{P Q_x}^\eta$ for all connected components $Q_x$ of $\tilde{S} \setminus \tilde{\lambda}$.
By the composition property of the shearing maps, it follows that $\varphi_{PQ}^\varepsilon = \varphi_{PQ}^\eta$ for all connected components $P,Q \subset \tilde{S} \setminus \tilde{\lambda}$.
By Proposition \ref{prop::shearing_injective}, it follows that $\varepsilon = \eta$, so the cataclysm deformation is injective.
\end{proof}
\begin{cor}
\label{cor::injective_SLn}
If $\rho$ is a Hitchin representation into $\mathrm{PSL}(n, \mathbb{R})$ or into $\SL{n}$, then the cataclysm deformation based at $\rho$ is injective.
\end{cor}
\begin{proof}
For Hitchin representations into $\mathrm{PSL}(n, \mathbb{R})$ and $\SL{n}$ for odd $n$, the stabilizer of any triple of flags is trivial and the claim follows from Corollary \ref{cor::stab_trivial_injective}.
For $\SL{n}$ with $n$ even, the claim follows by projecting to $\mathrm{PSL}(n, \mathbb{R})$ and using injectivity there.
\end{proof}

\printbibliography

\end{document}